\documentclass[aihp]{imsart}

\RequirePackage{amsthm,amsmath,amsfonts,amssymb}
\RequirePackage[numbers,sort&compress]{natbib}
\RequirePackage[colorlinks,citecolor=blue,urlcolor=blue]{hyperref}
\RequirePackage{graphicx}
\usepackage{indentfirst}

\startlocaldefs
\theoremstyle{plain}
\newtheorem{axiom}{Axiom}
\newtheorem{claim}[axiom]{Claim}
\newtheorem{theorem}{Theorem}[section]
\newtheorem{lemma}[theorem]{Lemma}
\newtheorem{corollary}{Corollary}

\newtheorem{definition}[theorem]{Definition}
\newtheorem{example}{Example}
\theoremstyle{remark}

\endlocaldefs
		
\usepackage{caption}
\begin{document}
	\begin{frontmatter}
		\title{Poisson stability of solutions for stochastic evolution equations driven by fractional Brownian motion}
		\runtitle{Poisson stability of solutions for SDE driven by fBm}

		\begin{aug}
			\author[A]{\inits{F.}\fnms{Xinze}~\snm{Zhang}\ead[label=e1]{xzzhang21@mails.jlu.edu.cn}},
			\author[A,B]{\inits{S.}\fnms{Yong}~\snm{Li}\ead[label=e2]{liyong@jlu.edu.cn}},
			\author[A]{\inits{S.}\fnms{Xue}~\snm{Yang}\ead[label=e3]{xueyang@jlu.edu.cn}} \thanksref{t1}.
			
			\address[A]{School of Mathematics, Jilin University, ChangChun, People's Republic of China}
			
			\address[B]{Center for Mathematics and Interdisciplinary Sciences,\\    ~~~ Northeast Normal University, ChangChun, People's Republic of China\\
			\printead{e1,e2,e3}}

			\thankstext{t1}{Corresponding author.}
		\end{aug}
		
		\begin{abstract}
			In this paper, we study the problem of Poisson stability of solutions for stochastic semi-linear evolution equation driven by fractional Brownian motion
			\begin{displaymath}
				\mathrm{d} X(t)= \left( AX(t) + f(t,X(t)) \right) \mathrm{d}t + g\left( t \right)\mathrm{d}B^H_{Q}(t), 
			\end{displaymath}
			where $A$ is an exponentially stable linear operator acting on a separable Hilbert space ${\mathbb{H}}$, coefficients $f$ and $g$ are Poisson stable in time, and $B^H_Q (t)$ is a $Q$-cylindrical fBm with Hurst index $H > 1/2$. First, we establish the existence and uniqueness of the solution for this equation. Then, we prove that under the condition where the functions $f$ and $g$ are sufficiently "small", the equation admits a solution that exhibits the same character of recurrence as $f$ and $g$. The discussion is further extended to the asymptotic stability of these Poisson stable solutions. Finally, we conducted a numerical simulation using an example to validate our results.
		\end{abstract}

		
		
	\end{frontmatter}
	
	\section*{Statements and Declarations}
	\begin{itemize}
		\item Ethical approval\\
		Not applicable.
		
		\item Competing interests\\
		The authors have no competing interests as defined by Springer, or other interests that might be perceived to influence the results and/or discussion reported in this paper.
		
		\item Authors' contributions\\
		X. Zhang wrote the main manuscript text and prepared Fig. $\ref{F.1}$ and Fig. $\ref{F.2}$. Y. Li and X. Yang provided ideas and explored specific research methods. All authors reviewed the manuscript.
		
		\item Funding\\
		The second author was supported by National Natural Science Foundation of China (Grant No. 12071175 and 12471183). The third author was supported by National Natural Science Foundation of China (Grant No. 12071175 and 12371191).
		
		\item Availability of data and materials\\
		Not applicable.
	\end{itemize}

	\section{Introduction}
	
	In this paper, we investigate the Poisson stability of solutions for stochastic evolution equations (SEEs) driven by fractional Brownian motion (fBm) in a separable Hilbert space $({\mathbb{H}}, \Vert \cdot \Vert_\mathbb{H})$, given by
	\begin{equation}\label{1.1}
		\begin{aligned}
			\mathrm{d} X(t)= \left( AX(t) + f(t,X(t)) \right) \mathrm{d}t + g\left( t \right)\mathrm{d}B^H_{Q}(t), 
		\end{aligned}
	\end{equation}
	where $A$ is an infinitesimal generator which generates a $C_0$-semigroup $\left\{ U(t) \right\}_{t \geq 0}$ acting on ${\mathbb{H}}$; the functions $f \in C_b(\mathbb{R} \times \mathbb{H}, L^2(\mathbb{P}, \mathbb{H}))$ and $g \in C_b(\mathbb{R}, L^0_Q(\mathbb{U}, \mathbb{H}))$ represent bounded continuous mappings, with $C_b$ representing the space of bounded continuous mappings; the operator $Q$ is a non-negative, self-adjoint, bounded operator on $\mathbb{H}$, and $\{B^H_{Q}(t)\}_{t\geq0}$ is an $\mathbb{H}$-valued fBm.

	In the field of evolution equations that delineate the physical world, the recurrence of solutions is one of the most concerned topics. As Poincaré initially proposed, the Poisson stable solution or motion of the system expresses the most general recurrence. Poisson stable functions take various forms, including stationary, periodic \cite{11}, quasi-periodic \cite{12}, Bohr almost periodic \cite{17}, almost automorphic \cite{20}, Birkhoff recurrent \cite{21}, Levitan almost periodic \cite{22}, almost recurrent in the sense of Bebutov \cite{8}, pseudo-periodic \cite{25}, pseudo-recurrent \cite{1}, Poisson stable \cite{26} functions, among others.
	
	For the problems of Poisson stable solutions, B. A. Shcherbakov systematically studied the problem of existence of Poisson stable solutions of the equation
	\begin{equation}\label{1.2}
		\begin{aligned}
			\dot{x} = f(t, x), \quad x\in \mathcal{B},
		\end{aligned}
	\end{equation}
	where the function $f$ is Poisson stable in $t \in \mathbb{R}$ uniformly with respect to $x$ on every compact subset from Banach space $\mathcal{B}$.
	
	In his works \cite{25,1,26,7,28,27}, B. A. Shcherbakov established a method of comparability of functions by character of their recurrence. He proved that, under certain conditions, at least one solution of equation $\eqref{1.2}$
	shares the recurrence properties of the function $f$, coining such solutions as compatible (or uniformly compatible). The insights of Shcherbakov have since been expanded upon by various authors, as seen in references \cite{29,62,34,32}.
	
	In recent years, one has also begun to study the Poisson stable solutions for SEEs. Now let us recall some work, which are closely relevant to this paper: Cheban and Liu \cite{38} studied the problem of Poisson stability for the semi-linear SDEs and discussed the asymptotic stability of these Poisson stable solutions; Riccardo Montalto \cite{61} established the existence and stability of small-amplitude, time-quasi-periodic solutions for the Navier-Stokes equations under a perturbative external force, within high Sobolev norms. Akebat et al. \cite{39} obtained the existence and uniqueness of almost periodic solutions in distribution to affine SDE driven by fBm; Lu and Yang \cite{40} proved the existence of Poisson stable solutions for stochastic functional evolution equations with infinite delay; Damanik et al. \cite{60} demonstrated the existence and uniqueness of spatially quasi-periodic solutions to the $p$-generalized KdV equation on the real line with quasi-periodic initial data. Other research and progress can be found in \cite{63,48,66}.
	
	Despite the fBm being a significant extension of classical Brownian motion, offering the distinct advantage of more accurately depicting processes in the real world with long-term memory and self-similarity, research into the Poisson stable solutions for SEEs driven by fBm is notably sparse. When systems are subjected to "rough" external disturbances, the mentioned properties, along with the Hölder continuity of its paths, naturally facilitate the use of fBm in modeling complex phenomena, including economics \cite{50}, solid wave motion \cite{51}, mathematical finance \cite{52}, biological motivation \cite{54}, and so on.

	This paper aims to explore the Poisson stability of solutions to SEEs driven by fBm, a topic previously confined to the study of periodic and almost periodic solutions \cite{39}. Building on the approach introduced in \cite{38}, we extend the investigation to include all types of recurrence discussed above. Notably, to the best of our knowledge, the results presented here are novel.
	
	One of the key challenges in this context is the absence of martingale and Markov properties in fBm. To address this, we developed a new mathematical framework by defining a novel space, norm, and the concept of $B^H_Q(t)$. These tools provide a more comprehensive description of the system's dynamics, effectively addressing the limitations of traditional methods when dealing with fBm-driven systems.
	
	From a theoretical perspective, we established specific criteria involving parameters, including the function $g$, to ensure the existence and uniqueness of solutions to the equation. By employing Shcherbakov's framework, the classical comparison principle, and various inequality techniques such as Hölder's inequality, Gronwall's inequality, and the Hardy-Littlewood inequality, we successfully addressed the non-standard properties of fBm. This approach not only guarantees the stability of the solutions but also reveals their long-term recurrence properties.

	The structure of the paper is as follows. In Section 2, we collect some known notions and facts. Namely we give some necessary preliminaries in the stochastic integration with respect to fBm, and we present the definitions of all important classes of Poisson stable functions and their basic properties. We also give a short survey of Shcherbakov’s results on comparability of functions by character of their recurrence.  In Section 3, we study Poisson stable solutions for stochastic semi-linear evolution equations. In Section 4, we dedicate to studying the dissipativity and the convergence of solutions to stochastic semi-linear evolution equations. In the last section, we provided an example to illustrate our main results.

	\section{Preliminaries}
	
	\subsection{The Spaces $C(\mathbb{R}, \mathbb{X})$ and $BUC(\mathbb{R} \times \mathbb{X}, \mathbb{X})$}
	
	Let $(\mathbb{X}, \rho)$ be a complete metric space. The space $C(\mathbb{R}, \mathbb{X})$ denotes the set of all continuous functions $\phi : \mathbb{R} \to \mathbb{X}$, equipped with the metric
	\[
	d(\phi_1, \phi_2) := \sup_{T > 0} \min \left\{ \max_{|t| \leq T}\rho(\phi_1(t), \phi_2(t)), \frac{1}{T} \right\}, \quad \phi_1, \phi_2 \in C(\mathbb{R}, \mathbb{X}).
	\]
	It is well-known that $(C(\mathbb{R}, \mathbb{X}), d)$ is a complete metric space (see [\citenum{1}, ChI] for details). Unless stated otherwise, convergence in $C(\mathbb{R}, \mathbb{X})$ refers to convergence with respect to this metric.
	
	\begin{claim}\label{remark 2.2}
		The following statements hold:
		\begin{enumerate}
			\item The metric $d$ induces the compact-open topology on $C(\mathbb{R}, \mathbb{X})$.
			\item The following conditions are equivalent:
			\begin{itemize}
				\item[(1)] $d(\phi_n, \phi) \to 0$ as $n \to \infty$;
				\item[(2)] $\lim\limits_{n \to \infty} \max\limits_{|t| < T} \rho(\phi_n(t), \phi(t)) = 0$ for every $T > 0$;
				\item[(3)] There exists a sequence $T_n \to \infty$ such that $\lim\limits_{n \to \infty} \max\limits_{|t| < T_n} \rho(\phi_n(t), \phi(t)) = 0$.
			\end{itemize}
		\end{enumerate}
	\end{claim}
	
	Let $BUC(\mathbb{R} \times \mathbb{X}, \mathbb{X})$ denote the space of all functions $F : \mathbb{R} \times \mathbb{X} \to \mathbb{X}$ that satisfy the following conditions:
	\begin{enumerate}
		\item $F$ is continuous in $t$, uniformly with respect to $x$ on every bounded subset $D \subseteq \mathbb{X}$;
		\item $F$ is bounded on every bounded subset of $\mathbb{R} \times \mathbb{X}$.
	\end{enumerate}
	
	For $F, G \in BUC(\mathbb{R} \times \mathbb{X}, \mathbb{X})$, consider a sequence of bounded subsets $\left\{ D_n \right\} \subseteq \mathbb{X}$ such that $D_n \subset D_{n + 1}$ for all $n \in \mathbb{N}$ and $\mathbb{X} = \bigcup\limits_{n \geq 1} D_n$. Define the metric
	\begin{equation}\label{2.5}
		d_{BUC}(F, G) := \sum_{n = 1}^{\infty} \frac{1}{2^n} \frac{d_n(F, G)}{1 + d_n(F, G)},
	\end{equation}
	where $d_n(F, G) := \sup\limits_{|t| \leq n, x \in D_n} \rho(F(t, x), G(t, x))$. It can be shown that $(BUC(\mathbb{R} \times \mathbb{X}, \mathbb{X}), d_{BUC})$ is a complete metric space, and that $d_{BUC}(F_n, F) \to 0$ if and only if $F_n(t, x) \to F(t, x)$ uniformly on every bounded subset of $\mathbb{R} \times \mathbb{X}$.
	
	Given $F \in BUC(\mathbb{R} \times \mathbb{X}, \mathbb{X})$ and $\tau \in \mathbb{R}$, define the translation $F^{\tau}$ by $F^{\tau}(t, x) := F(t + \tau, x)$ for $(t, x) \in \mathbb{R} \times \mathbb{X}$. Let $\bar{H}(F) := \overline{\left\{ F^{\tau} : \tau \in \mathbb{R} \right\}}$ represent the hull of F, where the closure is taken under the metric $d_{BUC}$ given by \eqref{2.5}. The mapping $\sigma: \mathbb{R} \times BUC(\mathbb{R} \times \mathbb{X}, \mathbb{X}) \to BUC(\mathbb{R} \times \mathbb{X}, \mathbb{X})$ defined by $\sigma(\tau, F) := F^{\tau}$ forms a dynamical system, satisfying $\sigma(0, F) = F$, $\sigma(\tau_1 + \tau_2, F) = \sigma(\tau_2, \sigma(\tau_1, F))$, and continuity in both arguments (see \cite{9} for details).
	
	Finally, let $BC(\mathbb{X}, \mathbb{X})$ denote the space of all bounded, continuous functions $F : \mathbb{X} \to \mathbb{X}$. For any $F, G \in BC(\mathbb{X}, \mathbb{X})$, define the metric
	\[
	d_{BC}(F, G) := \sum_{n = 1}^{\infty} \frac{1}{2^n} \frac{d_n(F, G)}{1 + d_n(F, G)},
	\]
	where $d_n(F, G) := \sup\limits_{x \in D_n} \rho(F(x), G(x))$. The space $(BC(\mathbb{X}, \mathbb{X}), d_{BC})$ is also a complete metric space.
	
	\subsection{Fractional Brownian Motion}
	
	We begin by introducing the fBm and the associated Wiener integral over the interval $[0, t]$. This definition will serve as the foundation for extending the integral to the interval $(- \infty, t]$, ensuring consistency and correctness in the extension process. Let $(\Omega, \mathcal{F}, \{\mathcal{F}_t\}_{t \geq 0}, \mathbb{P})$ be a complete filtered probability space. A two-sided one-dimensional fractional Brownian motion $\beta^H(t), t \in \mathbb{R}$, is a centered Gaussian process with covariance function
	\begin{displaymath}
		R_H(t, s) = \mathbb{E}[\beta^H(t)\beta^H(s)] = \frac{1}{2}(|t|^{2H} + |s|^{2H} - |t-s|^{2H}), \quad t, s \in \mathbb{R},
	\end{displaymath}
	where $H \in (0, 1)$ is the Hurst parameter, which characterizes the "memory" of the process.
	
	First, we define the Wiener integral with respect to one-dimensional fBm $\beta^H$ over the finite interval $[0, T]$. Let $T > 0$, and let $\Lambda$ denote the linear space of $\mathbb{R}$-valued step functions on $[0, T]$, given by
	\begin{displaymath}
		\varphi(t) = \sum_{i=1}^{n-1} x_i I_{[t_i, t_{i+1}]}(t),
	\end{displaymath}
	where $x_i \in \mathbb{R}, 0 = t_1 < t_2 < \cdots < t_n = T$, and $I$ is the indicator function. The Wiener integral of the step function $\varphi$ with respect to $\beta^H$ is defined as
	\begin{displaymath}
		\int_{0}^{T} \varphi(s) \,{\rm d}\beta^H(s) = \sum_{i=1}^{n-1} x_i (\beta^H(t_{i+1}) - \beta^H(t_i)).
	\end{displaymath}
	Let $\mathbb{U}$ be the Hilbert space defined as the closure of $\Lambda$ with respect to the inner product
	\begin{displaymath}
		\langle I_{[0,t]}, I_{[0,s]} \rangle_{\mathbb{U}} = R_H(t,s).
	\end{displaymath}
	The mapping
	\begin{displaymath}
		\varphi = \sum_{i=1}^{n-1} x_i I_{[t_i, t_{i+1}]} \mapsto \int_{0}^{T} \varphi(s) \,{\rm d}\beta^H(s)
	\end{displaymath}
	is an isometry between $\Lambda$ and the linear span of $\{\beta^H(t) : t \in [0, T]\}$, which extends to an isometry between $\mathbb{U}$ and the first Wiener chaos of the fBm $\overline{\text{span}}^{L^2(\Omega)} \{\beta^H(t) : t \in [0, T]\}$ (see \cite{2} for details). The image of an element $\phi \in \mathbb{U}$ under this isometry is called the Wiener integral of $\phi$ with respect to $\beta^H$.
	
	To give an explicit expression for this integral, consider the kernel function
	\begin{displaymath}
		K_H(t, s) = c_H s^{\frac{1}{2} - H} \int_{s}^{t} (u - s)^{H - \frac{3}{2}} u^{H - \frac{1}{2}} \,{\rm d}u, \quad t > s,
	\end{displaymath}
	where $c_H = \left(\frac{H(2H - 1)}{B(2 - 2H, H - \frac{1}{2})}\right)^{\frac{1}{2}}$, and $B$ denotes the Beta function. We can show that
	\begin{displaymath}
		\frac{\partial K_H}{\partial t}(t, s) = c_H \left(\frac{t}{s}\right)^{H - \frac{1}{2}} (t - s)^{H - \frac{3}{2}}.
	\end{displaymath}
	Define the linear operator $K_H^{\ast} : \Lambda \to L^2([0, T])$ as
	\begin{displaymath}
		(K_H^{\ast} \varphi)(s) = \int_{s}^{T} \varphi(t) \frac{\partial K_H}{\partial t}(t, s) \,{\rm d}t.
	\end{displaymath}
	Thus,
	\begin{displaymath}
		(K_H^{\ast} I_{[0,t]})(s) = K_H(t,s) I_{[0,t]}(s),
	\end{displaymath}
	and $K_H^{\ast}$ is an isometry between $\Lambda$ and $L^2([0, T])$, extendable to $\mathbb{U}$ (see \cite{3} for details).
	
	Now, consider the process $W = \{W(t), t \in [0, T]\}$, defined by
	\begin{displaymath}
		W(t) = \beta^H \left((K_H^{\ast})^{-1} I_{[0,t]}\right).
	\end{displaymath}
	One can show that $W$ is a Wiener process and that $\beta^H$ has the following Wiener integral representation:
	\begin{displaymath}
		\beta^H(t) = \int_{0}^{t} K_H(t,s) \,{\rm d}W(s).
	\end{displaymath}
	Moreover, for any $\phi \in \mathbb{U}$,
	\begin{displaymath}
		\int_{0}^{T} \phi(s) \,{\rm d}\beta^H(s) = \int_{0}^{T} (K_H^{\ast} \phi)(t) \,{\rm d}W(t),
	\end{displaymath}
	provided that $K_H^{\ast} \phi \in L^2([0, T])$.
	
	Let $L^2_{\mathbb{U}}([0, T]) = \{\phi \in \mathbb{U} : K_H^{\ast} \phi \in L^2([0, T])\}$. If $H > \frac{1}{2}$,  it follows from \cite{4} that
	\begin{equation}\label{2.1}
		L^{\frac{1}{H}}([0,T]) \subset L^2_{\mathbb{U}}([0,T]).
	\end{equation}
	
	After defining the stochastic integral on $[0, t]$, we extend it to $(- \infty, t]$ to accommodate broader applications while ensuring convergence and consistency. To maintain a well-defined integral over $(- \infty, t]$, the integrand $f$ must satisfy certain decay conditions as $s \to -\infty$, similar to the exponential decay seen in this paper with $e^{-\alpha (t - s)}$. Let $f : (-\infty, t] \to L^0_Q(\mathbb{U}, \mathbb{H})$ be a function that satisfies the following integrability condition:
	\begin{displaymath}
		\sum_{n=1}^{\infty} \int_{-\infty}^{t} |K^{\ast}_H(f Q^{\frac{1}{2}} e_n)(s)|^2 \, \mathrm{d}s < \infty.
	\end{displaymath}
	The stochastic integral of $f$ with respect to $B^H_Q(t)$ over $(- \infty, t]$ is then defined as:
	\begin{displaymath}
		\int_{-\infty}^{t} f(s) \, \mathrm{d}B^H_Q(s) := \sum_{n=1}^{\infty} \int_{-\infty}^{t} f(s) Q^{\frac{1}{2}} e_n \, \mathrm{d}\beta^H_n(s).
	\end{displaymath}
	This formulation ensures that the integral is well-defined for a broad class of functions, including those that exhibit the appropriate decay behavior as $s \to -\infty$. Such decay ensures the convergence of the integral even over the unbounded domain $(- \infty, t]$.
	
	Similar to the finite interval case, the integral can also be written as:
	\begin{displaymath}
		\int_{-\infty}^{t} f(s) \, \mathrm{d}B^H_Q(s) = \sum_{n=1}^{\infty} \int_{-\infty}^{t} K_H^{\ast}(f Q^{\frac{1}{2}}e_n)(s) \, \mathrm{d}W(s),
	\end{displaymath}
	where $K_H^{\ast}$ is the same kernel operator used for the finite interval case, and $W(s)$ is a Wiener process.
	
	By analogy with Theorem 1.1 in \cite{5}, we can similarly extend and derive the following lemma:
	
	\begin{lemma}\label{lemma 2.4}
		For any $\phi \in L^{\frac{1}{H}}((-\infty, T])$, we have
		\begin{displaymath}
			H(2H - 1)\int_{-\infty}^{T} \int_{-\infty}^{T} |\phi(r)| |\phi(u)| |r - u|^{2H - 2} \,{\rm d}r \,{\rm d}u \leq C_H \|\phi\|^2_{L^{\frac{1}{H}}((-\infty,T])},
		\end{displaymath}
	\end{lemma}
	 where $C_H$ is a positive constant depending on $H$.
	\begin{proof}
		Define the following two quantities:
		\begin{displaymath}
			\begin{aligned}
				&(I_{\alpha} \phi)(t) := \frac{1}{\Gamma(\alpha)} \int_{-\infty}^t (t - s)^{\alpha - 1} \phi(s) \, ds,
				\\
				&U(\phi) := H(2H - 1)\int_{-\infty}^{T} \int_{-\infty}^{T} |\phi(r)| |\phi(u)| |r - u|^{2H - 2} \,{\rm d}r \,{\rm d}u.
			\end{aligned}
		\end{displaymath}
		Applying Hölder's inequality, we obtain
		\begin{displaymath}
			U(\phi) \leq \alpha_H \left( \int_{-\infty}^T |\phi_r|^{\frac{1}{H}} dr \right)^H \left( \int_{-\infty}^T \left( \int_{-\infty}^T |\phi_u||r-u|^{2H-2} du \right)^{\frac{1}{1-H}} dr \right)^{1-H}.
		\end{displaymath}
		The second factor in the above expression, up to a multiplicative constant, is equal to the $\frac{1}{1-H}$ norm of the left-sided fractional integral $I_{\alpha} \phi$. Finally it suffices to apply the Hardy-Littlewood inequality (see \cite{65}, Theorem 1, p.119)
		\begin{displaymath}
			\|I_{\alpha} f\|_{L^{q}(-\infty,T)} \leq C(p,q,\alpha) \|f\|_{L^p(-\infty,T)},
		\end{displaymath}
		where $ 0 < \alpha < 1 $, and $ 1 < p < q < \infty $ satisfy the relation $ \frac{1}{q} = \frac{1}{p} - \alpha $. For the specific case of interest, we have $ \alpha = 2H - 1 $, $ q = \frac{1}{1 - H} $, and $ p = \frac{1}{H} $.
	\end{proof}
	
	Next, we consider an fBm with values in a Hilbert space and define the corresponding stochastic integral. Let $(\mathbb{U}, |\cdot|_\mathbb{U}, (\cdot,\cdot)_\mathbb{U})$ and $(\mathbb{H}, |\cdot|_\mathbb{H} , (\cdot,\cdot)_\mathbb{H})$ be two separable Hilbert spaces. Let $L(\mathbb{U}, \mathbb{H})$ denote the space of all bounded linear operators from $\mathbb{U}$ to $\mathbb{H}$, and let $Q \in L(\mathbb{U}, \mathbb{U})$ be a non-negative self-adjoint operator. Define $L^{\infty}_Q (\mathbb{U}, \mathbb{H})$ as the space of all $\xi \in L(\mathbb{U}, \mathbb{H})$ such that $\xi Q^{\frac{1}{2}}$ is a Hilbert-Schmidt operator, with norm
	\begin{displaymath}
		|\xi|^2_{L^0_Q (\mathbb{U}, \mathbb{H})} = |\xi Q^{\frac{1}{2}}|^2_{HS} = \text{tr}(\xi Q \xi^{\ast}).
	\end{displaymath}
	Such an operator $\xi : \mathbb{U} \to \mathbb{H}$ is called a $Q$-Hilbert-Schmidt operator.
	
	Let $\{\beta^H_n(t)\}_{n \in \mathbb{N}}$ be a sequence of independent two-sided one-dimensional standard fBm, and let $\{e_n\}_{n \in \mathbb{N}}$ be a complete orthonormal basis in $\mathbb{U}$. The series
	\begin{displaymath}
		\sum_{n = 1}^{\infty} \beta^H_n(t) e_n, \quad t \geq 0,
	\end{displaymath}
	does not necessarily converge in $\mathbb{U}$. Thus, we consider a $\mathbb{U}$-valued stochastic process $B^H_Q(t)$ given formally by
	\begin{equation}\label{fbm}
		B^H_Q(t) = \sum_{n = 1}^{\infty} \beta^H_n(t) Q^{\frac{1}{2}} e_n, \quad t \geq 0.
	\end{equation}
	When $Q$ is a non-negative self-adjoint trace class operator, this series converges in $\mathbb{U}$, and we define $B^H_Q(t)$ as a $\mathbb{U}$-valued $Q$-cylindrical fBm with covariance operator $Q$.
	
	Let $f: (-\infty, T] \to L^0_Q (\mathbb{U}, \mathbb{H})$ satisfy
	\begin{equation}\label{2.3}
		\sum_{n = 1}^{\infty} \|K^{\ast}_H(f Q^{\frac{1}{2}} e_n)\|_{L^2((-\infty,T];\mathbb{H})} < \infty.
	\end{equation}
	\begin{definition}
		If $f : (-\infty, T] \to L^0_Q (\mathbb{U}, \mathbb{H})$ satisfies \eqref{2.3}, its stochastic integral with respect to $B^H_Q$ is defined by
		\[
		\int_{-\infty}^{t} f(s) \,{\rm d}B^H_Q(s) := \sum_{n = 1}^{\infty} \int_{-\infty}^{t} f(s) Q^{\frac{1}{2}} e_n \,{\rm d}\beta^H_n(s) = \sum_{n = 1}^{\infty} \int_{-\infty}^{t} (K^{\ast}_H(f Q^{\frac{1}{2}}e_n))(s) \,{\rm d}W(s), \quad t \geq 0.
		\]
		If
		\begin{equation}\label{2.4}
			\sum_{n = 1}^{\infty} \|f Q^{\frac{1}{2}} e_n\|_{L^{\frac{1}{H}}((-\infty, T]; \mathbb{H})} < \infty,
		\end{equation}
		then \eqref{2.3} follows immediately from an extension of \eqref{2.1}.
	\end{definition}

	\subsection{Poisson stable functions}
	Let us recall the types of Poisson stable functions to be studied in this paper. For more details on these functions, see \cite{6, 7, 8}.
	\begin{definition}
		We say that a function $\varphi \in C(\mathbb{R},\mathbb{X})$ is stationary (respectively, $\tau$-periodic), if $\varphi(t) =\varphi(0)$ ${\rm (}$respectively, $\varphi(t + \tau) = \varphi(t)$${\rm )}$ for all $t \in \mathbb{R}$.
	\end{definition}
	
	\begin{definition}
		Let $\epsilon > 0$. We say that a number $\tau \in \mathbb{R}$ is $\epsilon$-almost period of the function $\varphi$, if $\rho\left(\varphi\left(t + \tau\right), \varphi(t)\right) < \epsilon$ for all $t \in \mathbb{R}$. Denote by $\mathcal{T} (\varphi, \epsilon)$ the set of $\epsilon$-almost periods of $\varphi$.
	\end{definition}
	
	\begin{definition}
		We say that a function $\varphi \in C(\mathbb{R}, \mathbb{X})$ is Bohr almost periodic, if the set of $\epsilon$-almost periods of $\varphi$ is relatively dense for each $\epsilon > 0$, i.e. for each $\epsilon > 0$ there exists $l = l(\epsilon) > 0$ such that $\mathcal{T} (\varphi, \epsilon) \cap [a, a + l]  \neq \emptyset$ for all $a \in \mathbb{R}$.
	\end{definition}
	
	\begin{definition}
		We say that a function $\varphi \in C(\mathbb{R}, \mathbb{X})$ is pseudo-periodic in the positive ${\rm (}$respectively, negative${\rm )}$ direction, if for each $\epsilon > 0$ and $l > 0$ there exists an $\epsilon$-almost period $\tau > l$ (respectively, $\tau < - l$) of the function $\varphi$. We say that a function $\varphi$ is pseudo-periodic if it is pseudo-periodic in both directions.
	\end{definition}
	
	\begin{definition}
		For given $\varphi \in C(\mathbb{R}, \mathbb{X})$, let $\varphi^h$ represent the $h$-translation of $\varphi$, i.e. $\varphi^h(t) = \varphi(h + t)$ for all $t \in \mathbb{R}$. The hull of $\varphi$, denoted by $\bar{H}(\varphi)$, is the set of all the limits of $\varphi^{h_n}$ in $C(\mathbb{R}, \mathbb{X})$, i.e.
		\begin{displaymath}
			\bar{H}(\varphi) := \left\{  \psi \in C(\mathbb{R}, \mathbb{X}) : \psi = \lim\limits_{n \to \infty} \varphi^{h_n} ~\text{for some sequence}~ \left\{ h_n \right\} \subset \mathbb{R} \right\}.
		\end{displaymath}

	\end{definition}
	It is well-known that the mapping $\pi(h, \varphi) := \varphi^h$ is a dynamical system from $\mathbb{R} \times C(\mathbb{R}, \mathbb{X})$ to $C(\mathbb{R}, \mathbb{X})$, i.e. $\pi(0, \varphi) = \varphi$, $\pi(h_1 + h_2, \varphi) = \pi(h_2,\pi(h_1,\varphi))$ and the mapping $\pi$ is continuous (see \cite{9}). In particular, the mapping $\pi$ restricted to $\mathbb{R} \times \bar{H}(\varphi)$ is a dynamical system.
	
	\begin{definition}
		We say that a number $\tau \in \mathbb{R}$ is an $\epsilon$-shift for $\varphi \in C(\mathbb{R}, \mathbb{X})$ if $d(\varphi^{\tau}, \varphi) < \epsilon$.
	\end{definition}
	
	\begin{definition}
		We say that a function $\varphi \in C(\mathbb{R}, \mathbb{X})$ is almost recurrent in the sense of Bebutov, if for every $\epsilon > 0$ the set $\left\{ \tau : d(\varphi^{\tau}, \varphi) < \epsilon \right\}$ is relatively dense.
	\end{definition}
	
	\begin{definition}
		We say that a function $\varphi \in C(\mathbb{R}, \mathbb{X})$ is Lagrange stable, if $\left\{ \varphi^h : h \in \mathbb{R} \right\}$ is a relatively compact subset of $C(\mathbb{R}, \mathbb{X})$.
	\end{definition}
	
	\begin{definition}
		We say that a function $\varphi \in C(\mathbb{R}, \mathbb{X})$ is Birkhoff recurrent, if it is almost recurrent and Lagrange stable.
	\end{definition}
	
	\begin{definition}
		We say that a function $\varphi \in C(\mathbb{R}, \mathbb{X})$ is Poisson stable in the positive (respectively, negative) direction, if for every $\epsilon > 0$ and $l > 0$ there exists $\tau > l$ ${\rm (}$respectively, $\tau < - l$${\rm )}$ such that $d(\varphi^{\tau}, \varphi) < \epsilon$. We say that a function $\varphi$ is Poisson stable if it is Poisson stable in both directions.
	\end{definition}

	In the following we will also denote $\mathbb{Y}$ as a complete metric space.
	
	\begin{definition}
		We say that a function $\varphi \in C(\mathbb{R}, \mathbb{X})$ is Levitan almost periodic, if there exists a Bohr almost periodic function $\psi \in C(\mathbb{R}, \mathbb{Y})$ such that for any $\epsilon > 0$ there exists $\delta = \delta(\epsilon) > 0$ such that $d(\varphi^{\tau}, \varphi) < \epsilon$ for all $\tau \in \mathcal{T}(\psi, \delta)$, where $\mathcal{T}(\psi, \delta)$ denotes the set of $\delta$-almost periods of $\psi$.
	\end{definition}
	
	\begin{definition}
		We say that a function $\varphi \in C(\mathbb{R}, \mathbb{X})$ is Bohr almost automorphic, if it is Levitan almost periodic and Lagrange stable.
	\end{definition}
	
	\begin{definition}
		We say that a function $\varphi \in C(\mathbb{R}, \mathbb{X})$ is quasi-periodic with the spectrum of frequencies $v_1, v_2,..., v_k$, if the following conditions are fulfilled:
		\begin{itemize}
			\item[1.] The numbers $v_1, v_2,..., v_k$ are rationally independent;
			\item[2.] There exists a continuous function $\phi : \mathbb{R}^k \to \mathbb{X}$ such that $\phi(t_1 + 2\pi, t_2 + 2\pi,..., t_k + 2\pi) = \phi(t_1, t_2,..., t_k)$ for all $(t_1, t_2,..., t_k) \in \mathbb{R}^k$;
			\item[3.] $\varphi(t) = \phi(v_1t, v_2t,..., v_kt)$ for all $t \in \mathbb{R}$.
		\end{itemize}
	\end{definition}
	
	\begin{definition}
		We say that a function $\varphi \in C(\mathbb{R}, \mathbb{X})$ is pseudo-recurrent, if for any $\epsilon > 0$ and $l \in \mathbb{R}$ there exists $L \geq l$ such that for any $\tau_0 \in \mathbb{R}$ we can find a number $\tau \in [l,L]$ satisfying
		\begin{displaymath}
			\sup\limits_{t \leq \frac{1}{\epsilon}} \rho(\varphi(t + \tau_0 + \tau), \varphi(t + \tau_0)) \leq \epsilon.
		\end{displaymath}
	\end{definition}
	
	Finally, we note that all types of functions introduced above, except for Lagrange stable functions, are Poisson stable.

	\subsection{Shcherbakov’s comparability method}
	\begin{definition}
		Let $\varphi \in C(\mathbb{R}, \mathbb{X})$ and $\mathfrak{N}_{\varphi}$ (respectively, $\mathfrak{M}_{\varphi}$) represent the family of all sequences $\left\{ t_n \right\} \subset \mathbb{R}$ such that $\varphi^{t_n}$ converges to $\varphi$ (respectively, $\left\{\varphi^{t_n} \right\}$ converges) in $C(\mathbb{R}, \mathbb{X})$ as $n \to \infty$. Additionally, let $\mathfrak{N}_{\varphi}^u$ (respectively, $\mathfrak{M}_{\varphi}^u$) represent the family of all sequences $\left\{ t_n \right\} \in \mathfrak{N}_{\varphi}$ (respectively, $\mathfrak{M}_{\varphi}^u$) such that $\varphi^{t + t_n}$ converges to $\varphi^{t}$ (respectively, $\varphi^{t + t_n}$ converges) uniformly in $t \in \mathbb{R}$ as $n \to \infty$.
	\end{definition}
	\begin{definition}
		We say that a function $\varphi \in C(\mathbb{R}, \mathbb{X})$ is comparable ${\rm (}$respectively, uniformly comparable${\rm )}$ by character of recurrence with $\phi \in C(\mathbb{R}, \mathbb{Y})$ if $\mathfrak{N}_{\phi} \subseteq \mathfrak{N}_{\varphi}$ ${\rm (}$respectively, $\mathfrak{M}_{\phi} \subseteq \mathfrak{M}_{\varphi}$${\rm )}$.
	\end{definition}
	
	\begin{theorem}\label{theorem 2.23}
		([\citenum{1}, ChII]) The following statements hold:
		\begin{itemize}
			\item[1.] $\mathfrak{M}_{\phi} \subseteq \mathfrak{M}_{\varphi}$ implies $\mathfrak{N}_{\phi} \subseteq \mathfrak{N}_{\varphi}$, and hence uniform comparability implies comparability;
			\item[2.] Let $\varphi \in C(\mathbb{R}, \mathbb{X})$ be comparable by character of recurrence with $\phi \in C(\mathbb{R}, \mathbb{Y})$. If the function $\phi$ is stationary ${\rm (}$respectively, $\tau$-periodic, Levitan almost periodic, almost recurrent, Poisson stable${\rm )}$, then so is $\varphi$;
			\item[3.] Let $\varphi \in C(\mathbb{R}, \mathbb{X})$ be uniformly comparable by character of recurrence with $\phi \in C(\mathbb{R}, \mathbb{Y})$. If the function $\phi$ is quasi-periodic with the spectrum of frequencies $v_1, v_2,..., v_k$ ${\rm (}$respectively, Bohr almost periodic, Bohr almost automorphic, Birkhoff recurrent, Lagrange stable${\rm )}$, then so is $\varphi$;
			\item[4.] Let $\varphi \in C(\mathbb{R}, \mathbb{X})$ be uniformly comparable by character of recurrence with $\phi \in C(\mathbb{R}, \mathbb{Y})$ and $\phi$ be Lagrange stable. If $\phi$ is pseudo-periodic ${\rm (}$respectively, pseudo-recurrent${\rm )}$, then so is $\varphi$.
		\end{itemize}
	\end{theorem}
	\begin{lemma}
		(\cite{38}) Let $\varphi \in C(\mathbb{R}, \mathbb{X})$, $\phi \in C(\mathbb{R}, \mathbb{Y})$ and $\mathfrak{M}^u_{\phi} \subseteq \mathfrak{M}^u_{\varphi}$. Then the following statements hold:
		\begin{itemize}
			\item[1.] $\mathfrak{N}^u_{\phi} \subseteq \mathfrak{N}^u_{\varphi}$;
			\item[2.] If the function $\phi$ is Bohr almost periodic, then $\varphi$ is also Bohr almost periodic.
		\end{itemize}
	\end{lemma}

	\begin{lemma}
		(\cite{38}) Let $F \in BUC(\mathbb{R} \times \mathbb{X}, \mathbb{X})$ and $\hat{F}(t) := F(t, \cdot)$ be a mapping from $\mathbb{R}$ to $BC(\mathbb{X}, \mathbb{X})$. We have
		\begin{itemize}
			\item[1.] $\mathfrak{M}_{F} = \mathfrak{M}_{\hat{F}}$ for any $F \in BUC(\mathbb{R} \times \mathbb{X}, \mathbb{X})$;
			\item[2.] $\mathfrak{M}^u_{F} = \mathfrak{M}^u_{\hat{F}}$ for any $F \in BUC(\mathbb{R} \times \mathbb{X}, \mathbb{X})$.
		\end{itemize}
		Here $\mathfrak{M}_{F}$ is the set of all sequences $\left\{ t_n \right\}$ such that $F^{t_n}$ converges in the space $BUC(\mathbb{R} \times \mathbb{X}, \mathbb{X})$ and $\mathfrak{M}^u_{F}$ is the set of all sequences $\left\{ t_n \right\}$ such that $F^{t + t_n}$ converges in the space $BUC(\mathbb{R} \times \mathbb{X}, \mathbb{X})$ uniformly with respect to $t \in \mathbb{R}$.
	\end{lemma}

	\subsection{Other important definitions}
	Let $\mathbb{H}$ be a real separable Hilbert space with the norm $\vert \cdot \vert$; $(\Omega,\mathcal{F},{\lbrace{\mathcal{F}_t}\rbrace}_{t\geq0},\mathbb{P})$ be a complete filtered probability space; $L^2(\mathbb{P}, \mathbb{H})$ be the space of $\mathbb{H}$-valued random variables $x$ such that
	\begin{displaymath}
		\begin{aligned}
			\mathbb{E}|x|^2:= \int_{\Omega} {|x|^2} \,{\rm d}\mathbb{P} < \infty.
		\end{aligned} \notag
	\end{displaymath}
	Then $L^2(\mathbb{P}, \mathbb{H})$ is a Hilbert space equipped with the norm
	\begin{displaymath}
		\begin{aligned}
			\Vert x\Vert_2:= (\int_{\Omega} {|x|^2} \,{\rm d}\mathbb{P})^{\frac{1}{2}}.
		\end{aligned} \notag
	\end{displaymath}
	Let $C_b(\mathbb{R}, L^2(\mathbb{P}, \mathbb{H}))$ represent the Hilbert space of all continuous and bounded mappings $\varphi : \mathbb{R} \to L^2(\mathbb{P}, \mathbb{H})$ equipped with the norm $||\varphi ||_{\infty} \!:= \sup\left\{ |\varphi(t)| : t \in \mathbb{R} \right\}$.
	\begin{claim}\label{remark 2.26}
		If $\varphi \in C_b(\mathbb{R}, L^2(\mathbb{P}, \mathbb{H}))$, then for any $\psi \in \bar{H}(\varphi)$ we have $||\psi(t)||_2 \leq ||\varphi||_{\infty}$ for all $t \in \mathbb{R}$.
	\end{claim}
	\begin{definition}
		We say that a semigroup of operators $\left\{ U(t) \right\}_{t \geq 0}$ is exponentially stable, if there are positive finite numbers $N, \alpha > 0$, such that $||U(t)|| \leq N e^{-\alpha t}$ for all $t \geq 0$.
	\end{definition}
	
	\begin{definition}
		We say that an $\mathcal{F}_t$-adapted processes $\left\{x(t)\right\}_{t \in \mathbb{R}}$ is a mild solution of equation $\eqref{1.1}$, if it satisfies the stochastic integral equation
		\begin{displaymath}
			\begin{aligned}
				x(t) = U(t - t_0)x(t_0) + \int_{t_0}^{t} {U(t - s)f(s, x(s)) }  \,{\rm d}s + \int_{t_0}^{t} {U(t - s)g(s)}  \,{\rm d}B^H_{Q}(s)
			\end{aligned} \notag
		\end{displaymath}
		for all $t \geq t_0$ and each $t_0 \in \mathbb{R}$.
	\end{definition}
	We adopt the framework of backward stochastic differential equations \cite{64} and derive the solution of the equation for $ t \leq t_0 $ based on a terminal condition $ x(t_0) = x(t_0) $. When the semigroup $ \left\{U(t)\right\}_{t \geq 0} $ generated by the operator $ A $ is exponentially stable, we can make reasonable assumption that $U(t - t_0)x(t_0) = 0$, as $ t_0 \to -\infty $. That is, the mild solution of equation $\eqref{1.1}$ can be written in the following form:
	\begin{displaymath}
		\begin{aligned}
			x(t) = \int_{-\infty}^{t} {U(t - s)f(s, x(s)) }  \,{\rm d}s + \int_{-\infty}^{t} {U(t - s)g(s) }  \,{\rm d}B^H_{Q}(s),
		\end{aligned} \notag
	\end{displaymath}
	and
	\begin{displaymath}
		\begin{aligned}
			x(t + T) &= \int_{-\infty}^{t + T} {U(t + T - s)f(s, x(s)) }  \,{\rm d}s + \int_{-\infty}^{t + T} {U(t + T - s)g(s)}  \,{\rm d}B^H_{Q}(s)\\
			&= \int_{-\infty}^{t} {U(t - 
			\tau)f(\tau + T, x(\tau + T)) }  \,{\rm d}\tau + \int_{-\infty}^{t} {U(t - \tau)g(\tau + T) }  \,{\rm d}\left(B^H_{Q}(\tau + T) - B^H_{Q}(T)\right)\\
			&= \int_{-\infty}^{t} {U(t - 
			\tau)f(\tau + T, x(\tau + T)) }  \,{\rm d}\tau + \int_{-\infty}^{t} {U(t - \tau)g(\tau + T) }  \,{\rm d}\tilde{B}^H_{Q}(\tau),
		\end{aligned} \notag
	\end{displaymath}
	where $\tilde{B}^H_{Q}(\tau) := B^H_{Q}(\tau + T) - B^H_{Q}(T)$. Due to the incremental stationarity of fBm, $\tilde{B}^H_{Q}(\tau)$ and $B^H_{Q}(s)$ have the same distribution. Therefore, when $f(s, x(s))$ and $g(s, x(s))$ have some kind of recurrence, we can study the recurrence of the solution to equation $\eqref{1.1}$ in distribution.

	Let $\mathcal{P}(\mathbb{H})$ be the space of all Borel probability measures on $\mathbb{H}$ endowed with the following metric:
	\begin{displaymath}
		\begin{aligned}
			\mathcal{D}(\mu, \nu) := \sup \bigg\{ \bigg\| {\int {f} \,{\rm d}\mu - \int {f} \,{\rm d}\nu } \bigg\| :||f||_{BL} \leq 1 \bigg\}, \quad \mu, \nu \in \mathcal{P}(\mathbb{H}),
		\end{aligned} \notag
	\end{displaymath}
	where $f \in \mathbb{H}$ are bounded Lipschitz continuous real-valued functions with the following norms:
	\begin{displaymath}
		\begin{aligned}
			||f||_{BL} = Lip(f) + ||f||_{\infty}, \quad Lip(f) = \sup\limits_{x \neq y} \frac{|f(x) - f(y)|}{|x - y|}, \quad ||f||_{\infty} = \sup\limits_{x \in \mathbb{H}} |f(x)|.
		\end{aligned} \notag
	\end{displaymath}
	We say that a sequence $\left\{ \mu_n \right\} \subset \mathbb{P}(\mathbb{H})$ weakly converge to $\mu$, if $\int {f} \,{\rm d}\mu_n \to \int {f} \,{\rm d}\mu$ for all $f \in C_b(\mathbb{H})$, where $C_b(\mathbb{H})$ denotes the space of all bounded continuous real-valued functions on $\mathbb{H}$. We can show that $(P(\mathbb{H}), \mathcal{D})$ is a separable complete metric space and that a sequence $\left\{ \mu_n \right\}$ weakly converges to $\mu$ if and only if $\mathcal{D}(\mu_n, \mu) \to 0$ as $n \to \infty$.
	\begin{definition}
		We say that a sequence of random variables $\left\{ x_n \right\}$ converge in distribution to the random variable $x$, if the corresponding laws $\left\{ \mu_n \right\}$ of $\left\{ x_n \right\}$ weakly converge to the law $\mu$ of $x$, i.e. $\mathcal{D}(\mu_n, \mu) \to 0$ as $n \to \infty$.
	\end{definition}
	\begin{definition}
		Let $\left\{ \varphi(t) \right\}_{t \in \mathbb{R}}$ be a mild solution of equation $\eqref{1.1}$. We say that $\varphi$ is compatible ${\rm (}$respectively, uniformly compatible${\rm )}$ in distribution if $\mathfrak{N}_{(f,g)} \subseteq \tilde{\mathfrak{N}}_{\varphi}$ ${\rm (}$respectively, $\mathfrak{M}_{(f,g)} \subseteq \tilde{\mathfrak{M}}_{\varphi}$${\rm )}$. Here, $\tilde{\mathfrak{N}}{\varphi}$ ${\rm (}$respectively, $\tilde{\mathfrak{M}}_{\varphi}$${\rm )}$ denotes the set of sequences $\left\{ t_n \right\} \subset \mathbb{R}$ such that the sequence $\left\{ \varphi(\cdot + t_n) \right\}$ converges to $\varphi(\cdot)$ ${\rm (}$respectively, $\left\{ \varphi(\cdot + t_n) \right\}$ converges ${\rm )}$ in distribution uniformly on any compact interval.
	\end{definition}

	\section{Semi-linear evolution equations driven by fBm}
	In this section, we consider the following stochastic semi-linear evolution equation:
	\begin{equation}\label{4.1}
		\begin{aligned}
			\mathrm{d} X(t)= \big( AX(t) + f(t,X(t)) \big) \mathrm{d}t + g(t)\mathrm{d}B^H_{Q}(t), 
		\end{aligned}
	\end{equation}
	where $A$ and $B^H_{Q}$ are the same as in equation $\eqref{1.1}$, and $f \in C(\mathbb{R} \times \mathbb{H}, L^2(\mathbb{P}, \mathbb{H})), g \in C(\mathbb{R}, L^0_Q(\mathbb{U}, \mathbb{H}))$ are $\mathcal{F}_t$-adapted.

	\begin{definition}
		We define some conditions regarding $f$ and $g$:
		\begin{itemize}
			\item[(C1)] There exists positive numbers $C_f, L \geq 0$ such that $|f(t, 0)| \leq C_f$ for any $t \in \mathbb{R}$ and $Lip(f) \leq L$, where
			\begin{displaymath}
				\begin{aligned}
					\text{Lip}(f) := \sup \left\{\frac{|f(t, x_1) - f(t, x_2)|}{|x_1 - x_2|} : x_1 \neq x_2, t \in \mathbb{R} \right\};
				\end{aligned} \notag
			\end{displaymath}
			\item[(C2)] There exists a positive number $C_{g}$ such that
			\begin{displaymath}
				\begin{aligned}
					\Vert g(t)\Vert_{L^0_Q(\mathbb{U}, \mathbb{H})} \leq C_{g}, \text{ uniformly in } \mathbb{R};
				\end{aligned}\notag
			\end{displaymath}
			\item[(C3)] $f$ is continuous in $t$ uniformly with respect to $x$ on each bounded subset $Q \subset \mathbb{H}$.
		\end{itemize}
	\end{definition}
	\begin{claim}~
		\begin{itemize}
			\item[1.] If $f$ and $g$ satisfy $(C1)$-$(C2)$ with the constants $C_f, C_g$ and $L$, then every pair of functions $(\tilde{f}, \tilde{g})$ in $H(f, g) := \overline{\left\{ (f^{\tau} ,g^{\tau}) : \tau \in \mathbb{R} \right\}} $, the hull of $(f, g)$, also posses the same property with the same constants;
			\item[2.]  If $f$ and $g$ satisfy the conditions $(C1)$-$(C3)$, then $f \in BUC(\mathbb{R} \times \mathbb{H}, L^2(\mathbb{P}, \mathbb{H})), g \in BUC(\mathbb{R}, L^0_Q(\mathbb{U}, \mathbb{H})$ and $H(f, g) \subset BUC(\mathbb{R} \times \mathbb{H}, L^2(\mathbb{P}, \mathbb{H})) \times BUC(\mathbb{R}, L^0_Q(\mathbb{U}, \mathbb{H}))$.
		\end{itemize}
	\end{claim}
	\begin{theorem}\label{theorem 4.3}
		Assume that the semigroup $\{U(t)\}_{t \geq 0}$ generated by the operator $A$ is exponentially stable, and that the functions $f \in C(\mathbb{R} \times \mathbb{H}, L^2(\mathbb{P}, \mathbb{H}))$ and $g \in C(\mathbb{R}, L^0_Q(\mathbb{U}, \mathbb{H}))$ satisfy conditions $(C1)$ and $(C2)$. If
		\begin{displaymath}
			\begin{aligned}
				\theta_1 :=  \frac{N^2 L^2}{{\alpha}^2} <1,
			\end{aligned} \notag
		\end{displaymath}
		then equation $\eqref{4.1}$ admits a unique bounded solution in $C_b(\mathbb{R}, L^2(\mathbb{P}, \mathbb{H}))$.
	\end{theorem}
	\begin{proof}
		As the proof is well-known, we provide only a brief version. Since the semigroup $\left\{ U(t) \right\}_{t \leq 0}$ generated by operator $A$ is exponentially stable, we can verify that $x \in C_b(\mathbb{R} ,L^2(\mathbb{P}, \mathbb{H}))$ is a mild solution of equation $\eqref{4.1}$ if and only if it satisfies the following integral equation:
		\begin{displaymath}
			\begin{aligned}
				x(t) = \int_{-\infty}^{t} {U(t - \tau)f(\tau, x(\tau))} \,{\rm d}\tau + \int_{-\infty}^{t} {U(t - \tau)g(\tau)} \,{\rm d}B^H_{Q}(\tau).
			\end{aligned} \notag
		\end{displaymath}
		We define operator $\mathcal{T}$ on $C_b(\mathbb{R}, L^2(\mathbb{P}, \mathbb{H}))$ as follows:
		\begin{displaymath}
			\begin{aligned}
				(\mathcal{T}x)(t) = \int_{-\infty}^{t} {U(t - \tau)f(\tau, x(\tau))} \,{\rm d}\tau + \int_{-\infty}^{t} {U(t - \tau)g(\tau)} \,{\rm d}B^H_{Q}(\tau).
			\end{aligned} \notag
		\end{displaymath}
		Since $f, g$ satisfy conditions $(C1)$ and $(C2)$, we can show that operator $\mathcal{T}$ maps $C_b(\mathbb{R},L^2(\mathbb{P}, \mathbb{H}))$ into itself. Note that
		\begin{align*}
			\mathbb{E}|(\mathcal{T}x_1)(t) - (\mathcal{T}x_2)(t)|^2
			 &\leq  \mathbb{E} \bigg| \int_{-\infty}^{t} {U(t - \tau) \big( f(\tau, x_1(\tau)) -f(\tau, x_2(\tau)) \big)} \,{\rm d}\tau \bigg|^2 
			\\ &\leq  \mathbb{E} \left( \int_{-\infty}^{t} {Ne^{-{\alpha}(t - \tau)} L | x_1(\tau)  - x_2(\tau)| } \,{\rm d}\tau \right)^2 
			\\ &\leq N^2 L^2  \int_{-\infty}^{t} {e^{-{\alpha}(t - \tau)} } \,{\rm d}\tau \cdot \mathbb{E} \int_{-\infty}^{t} {e^{-{\alpha}(t - \tau)} | x_1(\tau)  - x_2(\tau)|^2 } \,{\rm d}\tau
			\\ &\leq \frac{N^2 L^2}{{\alpha}^2}  \sup\limits_{\tau \in \mathbb{R}} \mathbb{E}|x_1(\tau) - x_2(\tau)|^2.
		\end{align*} 
		Therefore,
		\begin{displaymath}
			\begin{aligned}
				\sup\limits_{t \in \mathbb{R}} \mathbb{E}|(\mathcal{T}x_1)(t) - (\mathcal{T}x_2)(t)|^2 \leq \theta_1 \sup\limits_{t \in \mathbb{R}} \mathbb{E}|x_1(\tau) - x_2(\tau)|^2.
			\end{aligned} \notag
		\end{displaymath}
		Since $ \theta_1 < 1 $, the operator $ \mathcal{T} $ is a contraction mapping on $ C_b(\mathbb{R}, L^2(\mathbb{P}, \mathbb{H})) $. By the contraction mapping theorem, there exists a unique $ \xi \in C_b(\mathbb{R}, L^2(\mathbb{P}, \mathbb{H})) $ such that $ \mathcal{T}\xi = \xi $, which is the unique $ L^2 $-bounded solution to equation $\eqref{4.1}$.
	\end{proof}
	\begin{theorem}\label{theorem 4.5}
		Assume that the semigroup $\{U(t)\}_{t \geq 0}$ generated by the operator $A$ is exponentially stable, and that the functions $f \in C(\mathbb{R} \times \mathbb{H}, L^2(\mathbb{P}, \mathbb{H}))$ and $g \in C(\mathbb{R}, L^0_Q(\mathbb{U}, \mathbb{H}))$ satisfy conditions $(C1)$ and $(C2)$.
		Then we can obtain the following statements:
		\begin{itemize}
			\item[1.] If $ L < \frac{{\alpha}}{2N} $, equation $\eqref{4.1}$ has a unique solution $ \xi \in C(\mathbb{R}, B[0, R]) $, where $R$ is the larger of the two solutions of the following quadratic equation:
			\begin{displaymath}
				\left( 1 - \frac{2N^2}{\alpha^2} L^2 \right) R^2 - \frac{2N^2}{\alpha^2} 2C_f L R = \frac{2N^2}{\alpha^2} \left( C_f^2 + C_H C_g \alpha^{2 - 2H} (2H - 1)^{2H - 1} \right),
			\end{displaymath}
			and
			\begin{displaymath}
				\begin{aligned}
					B[0, R] := \left\{  x \in L^2(\mathbb{P}, \mathbb{H}) : ||x||_2 \leq R \right\};
				\end{aligned} \notag
			\end{displaymath}
			\item[2.] If $L < \frac{{\alpha}}{2\sqrt{2}N}$ and additionally $f$ satisfies condition $(C3)$, then
			\begin{itemize}
				\item[(a)] $\mathfrak{M}^u_{(F, G)} \subseteq \tilde{\mathfrak{M}}^u_{\xi}$, recalling that $\tilde{\mathfrak{M}}^u_{\xi}$ means the set of all sequences $\left\{ t_n \right\}$ such that $\xi(t + t_n)$ converges in distribution uniformly in $t \in \mathbb{R}$;
				\item[(b)] the solution $\xi$ is uniformly compatible in distribution.
			\end{itemize}
		\end{itemize}
	\end{theorem}
	\begin{proof}
		1. Note that $C(\mathbb{R}, B[0,R])$ is a complete metric space. By Theorem $\ref{theorem 4.3}$, it follows that the operator $\mathcal{T}$ maps the space $C_b(\mathbb{R},L^2(\mathbb{P}, \mathbb{H}))$ into itself. We now proceed to prove that
		\begin{displaymath}
			\begin{aligned}
				\mathcal{T} : C(\mathbb{R}, B[0,R]) \to C(\mathbb{R}, B[0,R]).
			\end{aligned} \notag
		\end{displaymath}
		Let $\varphi \in C(\mathbb{R}, B[0,R])$, then we have
		\begin{equation}
			\begin{aligned}\label{3.8}
				(\mathcal{T}\varphi)(t) &= \int_{-\infty}^{t} {U(t - \tau)f(\tau, \varphi(\tau))} \,{\rm d}\tau + \int_{-\infty}^{t} {U(t - \tau)g(\tau)} \,{\rm d}B^H_{Q}(\tau)\\
				&= I_1 + I_2.
			\end{aligned}
		\end{equation}
		For the first term $I_1$, using the Cauchy-Schwarz inequality we get
		\begin{equation}
			\begin{aligned}\label{3.9}
				\mathbb{E}|I_1|^2 &= \mathbb{E} \bigg| \int_{-\infty}^{t} {U(t - \tau)f(\tau, \varphi(\tau))} \,{\rm d}\tau \bigg|^2 
				\\ & \leq \mathbb{E} \bigg[ \int_{-\infty}^{t} {Ne^{-{\alpha}(t - \tau)} |f(\tau, \varphi(\tau))|} \,{\rm d}\tau \bigg]^2 
				\\ & \leq {N}^2  \int_{-\infty}^{t} {e^{-{\alpha}(t - \tau)} } \,{\rm d}\tau \int_{-\infty}^{t} {e^{-{\alpha}(t - \tau)} \mathbb{E}|f(\tau, \varphi(\tau))|^2} \,{\rm d}\tau 
				\\ & \leq \frac{{N}^2}{\alpha^2} \sup\limits_{\tau \in \mathbb{R}}\mathbb{E}|f(\tau, \varphi(\tau))|^2
				\\ & \leq \frac{(C_f + LR)^2 {N}^2}{\alpha^2}.
			\end{aligned}
		\end{equation}
		For the second term $I_2$, using Hölder inequality and lemma $\ref{lemma 2.4}$ yields
		\begin{align}\label{3.10}
			\notag \mathbb{E}|I_2|^2 &= \mathbb{E} \bigg| \int_{-\infty}^{t} {U(t - \tau)g(\tau)} \,{\rm d}B^H_{Q}(\tau) \bigg|^2 
			\\ \notag & = \sum_{n = 1}^{\infty} \mathbb{E} \bigg| \int_{-\infty}^{t} {U(t - \tau)g(\tau)Q^{\frac{1}{2}}e_n } \,{\rm d}\beta^H_{n}(\tau) \bigg|^2 
			\\\notag  & = \sum_{n = 1}^{\infty}H(2H - 1)  \int_{-\infty}^{t} \int_{-\infty}^{t} {\big| U(t - s)g(s)Q^{\frac{1}{2}}e_n \big| \cdot \big| U(t - \tau)g(\tau)Q^{\frac{1}{2}}e_n \big| }
			{ \cdot \big|s - \tau\big|^{2H - 2}} \,{\rm d}s  \,{\rm d}\tau 
			\\ & \leq C_H \sum_{n = 1}^{\infty} \bigg( \int_{-\infty}^{t} \big| {Ne^{-{\alpha}(t - \tau)}g(\tau)Q^{\frac{1}{2}}e_n \big|^{\frac{1}{H}} } \,{\rm d}\tau  \bigg)^{2H}
			\\\notag & \leq C_H N^2 \bigg( \int_{-\infty}^{t} \big| {e^{-\frac{{\alpha}}{2}(t - \tau)} \big|^{\frac{2}{2H-1}} } \,{\rm d}\tau  \bigg)^{2H - 1} \sum_{n = 1}^{\infty}\bigg( \int_{-\infty}^{t}  \big| e^{-\frac{{\alpha}}{2}(t - \tau)}{g(\tau)Q^{\frac{1}{2}}e_n} \big|^{2} \,{\rm d}\tau  \bigg)
			\\\notag & \leq C_H N^2 (\frac{2H - 1}{{\alpha}})^{2H - 1} \sum_{n = 1}^{\infty}\bigg( \int_{-\infty}^{t} e^{-{\alpha}(t - \tau)}  \big|{g(\tau)Q^{\frac{1}{2}}e_n} \big|^{2} \,{\rm d}\tau  \bigg)
			\\\notag & \leq C_H N^2 \frac{(2H - 1)^{2H - 1}}{{\alpha}^{2H}} \sup\limits_{\tau \in \mathbb{R}} \Vert g(\tau)\Vert^2_{L^0_Q(\mathbb{U}, \mathbb{H})}
			\\\notag & \leq C_H C_g N^2 \frac{(2H - 1)^{2H - 1}}{{\alpha}^{2H}}.
		\end{align}
		From $\eqref{3.8}$-$\eqref{3.10}$ we have
		\begin{align*}
			\mathbb{E}|(\mathcal{T}\varphi)(t)|^2 = 2\left(\mathbb{E}|I_1|^2 + \mathbb{E}|I_2|^2\right) \leq \frac{2 {N}^2}{\alpha^2} \left( (C_f + LR)^2 + C_H C_g {\alpha}^{2 - 2H}(2H - 1)^{2H - 1} \right).
		\end{align*}
		So,
		\begin{equation}
			\begin{aligned}
				\sup\limits_{t \in \mathbb{R}} \mathbb{E}|(\mathcal{T}\varphi)(t)| \leq \frac{\sqrt{2} {N}}{\alpha} \left( (C_f + LR)^2 + C_H C_g {\alpha}^{2 - 2H}(2H - 1)^{2H - 1} \right)^{\frac{1}{2}} = R.
			\end{aligned}\notag
		\end{equation}
		Thus $\mathcal{T}$ is well defined.
		
		The contraction property of the operator $ \mathcal{T} $ follows from Theorem $\ref{theorem 4.3}$. When $ L < \frac{{\alpha}}{2N} $, we have $ \theta_1 =  \frac{N^2 L^2}{{\alpha}^2} < 1 $. Thus, the operator $ \mathcal{T} $ is a contraction mapping. By the contraction mapping theorem, there exists a unique function $ \xi \in C(\mathbb{R}, B[0, R]) $ such that $ \mathcal{T}(\xi) = \xi $.
		
		2-(a). Let $\left\{ t_n \right\} \in \mathfrak{M}^u_{(f, g)}$. Then there exists $(\tilde{f}, \tilde{g}) \in \bar{H}(f, g)$ such that for any $R > 0$
		\begin{equation}\label{4.6}
			\begin{aligned}
				\sup\limits_{t \in \mathbb{R}, |x| \leq R} |f(t + t_n, x) - \tilde{f}(t, x)| \to 0
			\end{aligned} 
		\end{equation}
		and
		\begin{equation}\label{4.7}
			\begin{aligned}
				\sup\limits_{t \in \mathbb{R}} |g(t + t_n) - \tilde{g}(t)| \to 0
			\end{aligned} 
		\end{equation}
		as $n \to \infty$. Due to the definition of self adjoint operator $Q$, we can obtain
		\begin{equation}
			\begin{aligned}
				\sup\limits_{t \in \mathbb{R}}|g(t + t_n) - \tilde{g}(t)|^2_{L^0_Q(\mathbb{U}, \mathbb{H})} \to 0
			\end{aligned}\notag
		\end{equation}
		as $n \to \infty$. Consider the following two equations:
		\begin{equation}\label{4.8}
			\begin{aligned}
				\mathrm{d} X(t)= \big( AX(t) + f^{t_n}(t, X(t)) \big) \mathrm{d}t + g^{t_n}(t)\mathrm{d}B^H_{Q}(t)
			\end{aligned}
		\end{equation}
		and
		\begin{equation}\label{4.9}
			\begin{aligned}
				\mathrm{d} X(t)= \big( AX(t) + \tilde{f}(t, X(t)) \big) \mathrm{d}t + \tilde{g}(t)\mathrm{d}B^H_{Q}(t),
			\end{aligned}
		\end{equation}
		where $f^{t_n}(t, \cdot) := f(t + t_n, \cdot)$ and $g^{t_n}(t) := f(t + t_n)$. 
		
		Firstly, regarding the existence and uniqueness of the solution of the equation $\eqref{4.8}$(respectively, equation $\eqref{4.9}$), we can directly obtain it from part $(1)$ of this theorem. Then we will show that for any $t \in \mathbb{R}$, the solution $\left\{ x_n(t) \right\}_{n \geq 1}$ of equation $\eqref{4.8}$ uniformly converges to the solution $\tilde{x}(t)$ of equation $\eqref{4.9}$ in $L^2$ norm.
		
		We can check that $\xi_n := x_n - \tilde{x}$ is the unique solution from $C(\mathbb{R}, B[0, 2R])$ of the equation
		\begin{equation}\label{4.10}
			\begin{aligned}
				\mathrm{d} X(t)= \big( AX(t) + f^{t_n}(t, x_n(t)) - \tilde{f}(t, \tilde{x}(t)) \big) \mathrm{d}t + \big( g^{t_n}(t) -  \tilde{g}(t) \big)\mathrm{d}B^H_{Q}(t),
			\end{aligned}
		\end{equation}
		where $f^{t_n} - \tilde{f} \in C_b(\mathbb{R} \times \mathbb{H}, L^2(\mathbb{R}\times \mathbb{H}, \mathbb{H})), g^{t_n} - \tilde{g} \in C_b(\mathbb{R}, L^0_Q(\mathbb{U}, \mathbb{H}))$. Similar to the proof of part $(1)$ of this theorem, we have
		\begin{align*}
			\mathbb{E}|\xi_n|^2 & = \mathbb{E} \bigg| \int_{-\infty}^{t} {U(t - \tau) (f^{t_n}(\tau, x_n(\tau)) - \tilde{f}(\tau,\tilde{x}(\tau)))} \,{\rm d}\tau + \int_{-\infty}^{t} {U(t - \tau) (g^{t_n}(\tau) -  \tilde{g}(\tau))} \,{\rm d}B^H_{Q}(\tau) \bigg|^2 
			\\ & \leq 2 \mathbb{E} \bigg| \int_{-\infty}^{t} {U(t - \tau) (f^{t_n}(\tau, x_n(\tau)) - \tilde{f}(\tau, \tilde{x}(\tau)))} \,{\rm d}\tau \bigg|^2  + 2 \mathbb{E} \bigg| \int_{-\infty}^{t} {U(t - \tau) (g^{t_n}(\tau) -  \tilde{g}(\tau))} \,{\rm d}B^H_{Q}(\tau) \bigg|^2 
			\\ & := I_1 + I_2.
		\end{align*}
		For the first item $I_1$, by the Cauchy-Schwarz inequality we have
		\begin{align*}
			I_1 &= 2 \mathbb{E} \bigg| \int_{-\infty}^{t} {U(t - \tau) (f^{t_n}(\tau, x_n(\tau)) - \tilde{f}(\tau,\tilde{x}(\tau)))} \,{\rm d}\tau \bigg|^2 
			\\ & \leq 2N^2  \int_{-\infty}^{t} {e^{-{\alpha}(t - \tau)} } \,{\rm d}\tau \int_{-\infty}^{t} {e^{-{\alpha}(t - \tau)} \mathbb{E}|f^{t_n}(\tau, x_n(\tau)) - \tilde{f}(\tau,\tilde{x}(\tau))|^2} \,{\rm d}\tau 
			\\ & \leq \frac{2 N^2}{{\alpha}^2}\sup\limits_{\tau \in \mathbb{R}} \mathbb{E} |f^{t_n}(\tau, x_n(\tau)) - \tilde{f}(\tau,\tilde{x}(\tau))|^2
			\\ & \leq \frac{2 N^2}{{\alpha}^2} \sup\limits_{\tau \in \mathbb{R}}\mathbb{E} \big| f^{t_n}(\tau, x_n(\tau)) - f^{t_n}(\tau, \tilde{x}(\tau)) + f^{t_n}(\tau, \tilde{x}(\tau)) - \tilde{f}(\tau, \tilde{x}(\tau))  \big|^2
			\\ & \leq \frac{4N^2}{{\alpha}^2} \sup\limits_{\tau \in \mathbb{R}} \left( \mathbb{E} \big| f^{t_n}(\tau, x_n(\tau)) - f^{t_n}(\tau, \tilde{x}(\tau)) \big|^2 + \mathbb{E} \big| f^{t_n}(\tau, \tilde{x}(\tau)) - \tilde{f}(\tau, \tilde{x}(\tau))  \big|^2 \right)
			\\ & \leq \frac{4N^2}{{\alpha}^2} \left(L^2 \sup\limits_{\tau \in \mathbb{R}}\mathbb{E} \big| x_n(\tau) - \tilde{x}(\tau) \big|^2 + \sup\limits_{\tau \in \mathbb{R}} \mathbb{E} \big| f^{t_n}(\tau, \tilde{x}(\tau)) - \tilde{f}(\tau, \tilde{x}(\tau))  \big|^2 \right)
			\\ & \leq \frac{4N^2}{{\alpha}^2} \left(L^2 \sup\limits_{\tau \in \mathbb{R}} \mathbb{E} | \xi_n |^2 + \sup\limits_{\tau \in \mathbb{R}} \mathbb{E} F^2_{n,\tau} \right),
		\end{align*}
		where
		\begin{align*}
			F_{n,\tau} := \big| f^{t_n}(\tau, \tilde{x}(\tau) - \tilde{f}(\tau, \tilde{x}(\tau))  \big|.
		\end{align*}
		For the second item $I_2$, we have
		\begin{align*}
			I_2 &= 2\mathbb{E} \bigg| \int_{-\infty}^{t} {U(t - \tau) (g^{t_n}(\tau) -  \tilde{g}(\tau))} \,{\rm d}B^H_{Q}(\tau) \bigg|^2
			\\  & = 2\sum_{n = 1}^{\infty} \mathbb{E} \bigg| \int_{-\infty}^{t} {U(t - \tau) (g^{t_n}(\tau) -  \tilde{g}(\tau)) Q^{\frac{1}{2}}e_n } \,{\rm d}\beta^H_{n}(\tau) \bigg|^2 
			\\ & \leq 2C_H \sum_{n = 1}^{\infty} \bigg( \int_{-\infty}^{t} \big| {Ne^{-{\alpha}(t - \tau)} (g^{t_n}(\tau) -  \tilde{g}(\tau)) Q^{\frac{1}{2}}e_n \big|^{\frac{1}{H}} } \,{\rm d}\tau  \bigg)^{2H}
			\\ & \leq 2C_H N^2 \bigg( \int_{-\infty}^{t} \big| {e^{-\frac{{\alpha}}{2}(t - \tau)} \big|^{\frac{2}{2H-1}} } \,{\rm d}\tau  \bigg)^{2H - 1} \sum_{n = 1}^{\infty}\bigg( \int_{-\infty}^{t}  \big| e^{-\frac{{\alpha}}{2}(t - \tau)}{(g^{t_n}(\tau) -  \tilde{g}(\tau)) Q^{\frac{1}{2}}e_n} \big|^{2} \,{\rm d}\tau  \bigg)
			\\ & \leq 2C_H N^2 (\frac{2H - 1}{{\alpha}})^{2H - 1} \sum_{n = 1}^{\infty}\bigg( \int_{-\infty}^{t} e^{-{\alpha}(t - \tau)}  \big|{(g^{t_n}(\tau) -  \tilde{g}(\tau)) Q^{\frac{1}{2}}e_n} \big|^{2} \,{\rm d}\tau  \bigg)
			\\ & \leq 2C_H N^2 \frac{(2H - 1)^{2H - 1}}{{\alpha}^{2H}} \sup\limits_{\tau \in \mathbb{R}} \Vert g^{t_n}(\tau) -  \tilde{g}(\tau)\Vert^2_{L^0_Q(\mathbb{U}, \mathbb{H})}
			\\ & \leq 2C_H N^2 \frac{(2H - 1)^{2H - 1}}{{\alpha}^{2H}} \sup\limits_{\tau \in \mathbb{R}} \mathbb{E} G^2_{n,\tau},
		\end{align*}
		where
		\begin{displaymath}
			\begin{aligned}
				G_{n,\tau} := \Vert g^{t_n}(\tau) -  \tilde{g}(\tau)\Vert_{L^0_Q(\mathbb{U}, \mathbb{H})}.
			\end{aligned}
		\end{displaymath}
		Therefore, we have 
		\begin{displaymath}
			\begin{aligned}
				\sup\limits_{t \in \mathbb{R}} |\xi_n|^2 \leq \frac{4N^2}{{\alpha}^2} \left(L^2 \sup\limits_{\tau \in \mathbb{R}} \mathbb{E} | \xi_n |^2 + \sup\limits_{\tau \in \mathbb{R}} \mathbb{E} F^2_{n,\tau} \right) + 2C_H N^2 \frac{(2H - 1)^{2H - 1}}{{\alpha}^{2H}} \sup\limits_{\tau \in \mathbb{R}} \mathbb{E} G^2_{n,\tau}.
			\end{aligned}
		\end{displaymath}
		Consequently,
		\begin{equation}\label{4.11}
			\begin{aligned}
				\left( 1 - \frac{4N^2 L^2}{{\alpha}^2} \right)\sup\limits_{\tau \in \mathbb{R}} \mathbb{E} |\xi_n|^2 
				\leq \frac{4N^2}{{\alpha}^2} \sup\limits_{\tau \in \mathbb{R}} \mathbb{E} F^2_{n,\tau} + 2C_H N^2 \frac{(2H - 1)^{2H - 1}}{{\alpha}^{2H}} \sup\limits_{\tau \in \mathbb{R}} \mathbb{E} G^2_{n,\tau}.
			\end{aligned}
		\end{equation}
		Due to $L < \frac{{\alpha}}{2N}$, the coefficient of $\sup\limits_{\tau \in \mathbb{R}} \mathbb{E} |\xi_n|^2$ is positive. It is easy to note that the families
		\begin{displaymath}
			\begin{aligned}
				\left\{ |\tilde{x}(\tau)|^2 : \tau \in \mathbb{R} \right\} \quad \text{and} \quad \left\{ |{x}_n(\tau)|^2 :n \in \mathbb{N}, \tau \in \mathbb{R} \right\}
			\end{aligned}
		\end{displaymath}
		are uniformly integrable, and according to conditions $(C1)$ and $(C2)$, the families
		\begin{displaymath}
			\begin{aligned}
				\left\{ F_{n,\tau}^2 : n \in \mathbb{N}, \tau \in \mathbb{R} \right\} \quad \text{and} \quad \left\{ G_{n,\tau}^2 :n \in \mathbb{N}, \tau \in \mathbb{R} \right\}
			\end{aligned}
		\end{displaymath}
		are uniformly integrable. These together with $\eqref{4.6}$ and $\eqref{4.7}$ imply that taking the limit in $\eqref{4.11}$, we obtain $x_n(t) \to \tilde{x}(t)$ uniformly in $t \in \mathbb{R}$ in $L^2$-norm. Since $L^2$ convergence implies convergence in distribution, we have $x_n(t) \to \tilde{x}(t)$ in distribution uniformly on $\mathbb{R}$. Thus we have $\left\{ t_n \right\} \in \mathfrak{M}^u_{x}$.
		
		2-(b). Let $\left\{ t_n \right\} \in \mathfrak{M}_{(f, g)}$. Then there exists $( \tilde{f}, \tilde{g}) \in \bar{H}(f, g)$ such that for any $R, T > 0$,
		\begin{equation}\label{4.12}
			\begin{aligned}
				\sup\limits_{|t| \leq T, |x| \leq R} |f(t + t_n, x) - \tilde{f}(t, x)| \to 0
			\end{aligned} 
		\end{equation}
		and
		\begin{equation}\label{4.13}
			\begin{aligned}
				\sup\limits_{|t| \leq T} |g(t + t_n) - \tilde{g}(t)| \to 0
			\end{aligned} 
		\end{equation}
		as $n \to \infty$. Likewise, we have
		\begin{equation}
			\begin{aligned}
				\sup\limits_{|t| \leq T}|g(t + t_n) - \tilde{g}(t)|^2_{L^0_Q(\mathbb{U}, \mathbb{H})} \to 0
			\end{aligned}\notag
		\end{equation}
		as $n \to \infty$. Similar to the proof of (ii)-(a): let $x_n$ and $\tilde{x}$ be the unique bounded solutions of equation $\eqref{4.8}$ and equation $\eqref{4.9}$ respectively, and $\xi_n := x_n - \tilde{x}$ represent the unique bounded solution of equation $\eqref{4.10}$. Next, we will prove that $\xi_n \to 0$ in the space $C(\mathbb{R}, L^2(\mathbb{R},\mathbb{H}))$, i.e. $\lim\limits_{n \to \infty} \max\limits_{|t| \leq T} \mathbb{E}|\xi_n(t)|^2= 0$ for any $T > 0$.
		
		Let $T > 0, t \in [-T, T], \tilde{t} > T$. We have the following estimate for the solution $\xi_n$ of equation $\eqref{4.10}$:
		\begin{displaymath}
			\begin{aligned}
				\mathbb{E}|\xi_n|^2 & = \mathbb{E} \bigg| \int_{-\infty}^{t} {U(t - \tau) (f^{t_n}(\tau, x_n(\tau)) - \tilde{f}(\tau,\tilde{x}(\tau)))} \,{\rm d}\tau + \int_{-\infty}^{t} {U(t - \tau) (g^{t_n}(\tau) -  \tilde{g}(\tau))} \,{\rm d}B^H_{Q}(\tau) \bigg|^2 
				\\ & \leq 2 \left[ \mathbb{E} \bigg| \int_{-\infty}^{t} {Ne^{-{\alpha}(t - \tau)} (f^{t_n}(\tau, x_n(\tau)) - \tilde{f}(\tau,\tilde{x}(\tau)))} \,{\rm d}\tau \bigg|^2  
				\! + \mathbb{E} \bigg| \int_{-\infty}^{t} {Ne^{-{\alpha}(t - \tau)} (g^{t_n}(\tau) -  \tilde{g}(\tau))} \,{\rm d}B^H_{Q}(\tau) \bigg|^2  \right]
				\\ & \leq 4N^2 \left[ \mathbb{E} \bigg| \int_{-\tilde{t}}^{t} {e^{-{\alpha}(t - \tau)} (f^{t_n}(\tau, x_n(\tau)) - \tilde{f}(\tau,\tilde{x}(\tau)))} \,{\rm d}\tau \bigg|^2  
				\! +  \mathbb{E}\bigg| \int_{-\infty}^{-\tilde{t}} {e^{-{\alpha}(t - \tau)} (f^{t_n}(\tau, x_n(\tau)) - \tilde{f}(\tau,\tilde{x}(\tau)))} \,{\rm d}\tau \bigg|^2 \right.
				\\ & \quad \left.+ \mathbb{E} \bigg| \int_{-\tilde{t}}^{t} {e^{-{\alpha}(t - \tau)} (g^{t_n}(\tau) -  \tilde{g}(\tau))} \,{\rm d}B^H_{Q}(\tau) \bigg|^2  
				\! + \mathbb{E}\bigg| \int_{-\infty}^{-\tilde{t}} {e^{-{\alpha}(t - \tau)} (g^{t_n}(\tau) -  \tilde{g}(\tau))} \,{\rm d}B^H_{Q}(\tau) \bigg|^2 \right]	
				\\ &:= 4N^2(J_1 + J_2 + J_3 + J_4).
			\end{aligned}
		\end{displaymath}
		For the first item $J_1$,
		\begin{displaymath}
			\begin{aligned}
				J_1 &=  \mathbb{E} \bigg| \int_{-\tilde{t}}^{t} {e^{-{\alpha}(t - \tau)} (f^{t_n}(\tau, x_n(\tau)) - \tilde{f}(\tau,\tilde{x}(\tau)))} \,{\rm d}\tau \bigg|^2  
				\\ & \leq \bigg| \int_{-\tilde{t}}^{t} {e^{-{\alpha}(t - \tau)}} \,{\rm d}\tau \bigg|^2 \sup\limits_{|t| \leq \tilde{t}}\mathbb{E}|f^{t_n}(\tau, x_n(\tau)) - \tilde{f}(\tau,\tilde{x}(\tau))|^2 
				\\ & \leq \frac{1}{\alpha^2} {\left(1 - e^{-{\alpha}(t + \tilde{t})}\right)}^2 \sup\limits_{|t| \leq \tilde{t}}\mathbb{E}|f^{t_n}(\tau, x_n(\tau)) - \tilde{f}(\tau,\tilde{x}(\tau))|^2.
			\end{aligned}
		\end{displaymath}
		For the second item $J_2$,
		\begin{align*}
			J_2 &= \mathbb{E}\bigg| \int_{-\infty}^{-\tilde{t}} {e^{-{\alpha}(t - \tau)} (f^{t_n}(\tau, x_n(\tau)) - \tilde{f}(\tau,\tilde{x}(\tau)))} \,{\rm d}\tau \bigg|^2
			\\ & \leq \bigg| \int_{-\infty}^{-\tilde{t}} {e^{-{\alpha}(t - \tau)} } \,{\rm d}\tau \bigg|^2 \sup\limits_{t \in \mathbb{R}}\mathbb{E}|f^{t_n}(\tau, x_n(\tau)) - \tilde{f}(\tau,\tilde{x}(\tau))|^2 
			\\ & \leq \frac{1}{\alpha^2} {e^{-2{\alpha}(t + \tilde{t})}} \sup\limits_{t \in \mathbb{R}}\mathbb{E}|f^{t_n}(\tau, x_n(\tau)) - \tilde{f}(\tau,\tilde{x}(\tau))|^2.
		\end{align*}
		For the third item $J_3$,
		\begin{align*}
			J_3 &= \mathbb{E} \bigg| \int_{-\tilde{t}}^{t} {e^{-{\alpha}(t - \tau)} (g^{t_n}(\tau) -  \tilde{g}(\tau))} \,{\rm d}B^H_{Q}(\tau) \bigg|^2
			\\ & \leq C_H \sum_{n = 1}^{\infty} \bigg( \int_{-\tilde{t}}^{t} \big| {e^{-{\alpha}(t - \tau)} (g^{t_n}(\tau) -  \tilde{g}(\tau)) Q^{\frac{1}{2}}e_n \big|^{\frac{1}{H}} } \,{\rm d}\tau  \bigg)^{2H}
			\\ & \leq C_H \bigg( \int_{-\tilde{t}}^{t} \big| {e^{-\frac{{\alpha}}{2}(t - \tau)} \big|^{\frac{2}{2H-1}} } \,{\rm d}\tau  \bigg)^{2H - 1} \sum_{n = 1}^{\infty}\bigg( \int_{-\infty}^{t}  \big| e^{-\frac{{\alpha}}{2}(t - \tau)}{(g^{t_n}(\tau) -  \tilde{g}(\tau)) Q^{\frac{1}{2}}e_n} \big|^{2} \,{\rm d}\tau  \bigg)
			\\ & \leq C_H \left( \frac{2H - 1}{{\alpha}} \left( 1 - e^{\frac{\alpha(\tilde{t} - t)}{2H - 1}}\right) \right)^{2H - 1} \sum_{n = 1}^{\infty}\bigg( \int_{-\tilde{t}}^{t} e^{-{\alpha}(t - \tau)}  \big|{(g^{t_n}(\tau) -  \tilde{g}(\tau)) Q^{\frac{1}{2}}e_n} \big|^{2} \,{\rm d}\tau  \bigg)
			\\ & \leq C_H \frac{(2H - 1)^{2H - 1}}{{\alpha}^{2H}} \sup\limits_{|t| \leq \tilde{t}} \Vert g^{t_n}(\tau) -  \tilde{g}(\tau)\Vert^2_{L^0_Q(\mathbb{U}, \mathbb{H})}
			\\ & \leq C(\alpha,H) \sup\limits_{|t| \leq \tilde{t}} \mathbb{E} G^2_{n,\tau}.
		\end{align*}
		For the fourth item $J_4$,
		\begin{align*}
			J_4 &= \mathbb{E}\bigg| \int_{-\infty}^{-\tilde{t}} {e^{-{\alpha}(t - \tau)} (g^{t_n}(\tau) -  \tilde{g}(\tau))} \,{\rm d}B^H_{Q}(\tau) \bigg|^2
			\\  & = e^{-2{\alpha}(t  + \tilde{t})} \sum_{n = 1}^{\infty} \mathbb{E} \bigg| \int_{-\infty}^{-\tilde{t}} {e^{-{\alpha}(-\tilde{t} - \tau)} (g^{t_n}(\tau) -  \tilde{g}(\tau)) Q^{\frac{1}{2}}e_n } \,{\rm d}\beta^H_{n}(\tau) \bigg|^2 
			\\ & \leq e^{-2{\alpha}(t  + \tilde{t})} C_H  \frac{(2H - 1)^{2H - 1}}{{\alpha}^{2H}} \sup\limits_{\tau \in \mathbb{R}} \Vert g^{t_n}(\tau) -  \tilde{g}(\tau)\Vert^2_{L^0_Q(\mathbb{U}, \mathbb{H})}
			\\ & \leq C(\alpha,H) e^{-2{\alpha}(t  + \tilde{t})} \sup\limits_{\tau \in \mathbb{R}} \mathbb{E} G^2_{n,\tau}.
		\end{align*}
		Therefore,
		\begin{displaymath}
			\begin{aligned}
				\mathbb{E}|\xi_n|^2 & \leq 4N^2 \bigg[ \frac{1}{v^2} {(1 - e^{-{\alpha}(t + \tilde{t})})}^2 \sup\limits_{|t| \leq \tilde{t}}\mathbb{E}|f^{t_n}(\tau, x_n(\tau)) - \tilde{f}(\tau,\tilde{x}(\tau))|^2 
				\\ & \quad  +  \frac{1}{{\alpha}^2} {e^{-2{\alpha}(t + \tilde{t})}} \sup\limits_{t \in \mathbb{R}}\mathbb{E}|f^{t_n}(\tau, x_n(\tau)) - \tilde{f}(\tau,\tilde{x}(\tau))|^2 
				\\ & \quad + C(\alpha,H) \sup\limits_{|t| \leq \tilde{t}} \mathbb{E} G^2_{n,\tau} + C(\alpha,H) e^{-2{\alpha}(t  + \tilde{t})} \sup\limits_{\tau \in \mathbb{R}} \mathbb{E} G^2_{n,\tau} \bigg].
			\end{aligned}
		\end{displaymath}
		Form the proof of 2-(a), we have
		\begin{equation}\label{4.16}
			\begin{aligned}
				\sup\limits_{\tau \in \mathbb{R}} \mathbb{E}|f^{t_n}(\tau, x_n(\tau)) - \tilde{f}(\tau,\tilde{x}(\tau))|^2 
				\leq 2 \sup\limits_{\tau \in \mathbb{R}} \big( \mathbb{E} \big| f^{t_n}(\tau, \xi_n(\tau)) \big|^2 + \mathbb{E} \big| \tilde{f}(\tau, \tilde{\xi}(\tau))  \big|^2 \big)
				\leq 4(C_f + LR)^2
			\end{aligned} \notag
		\end{equation}
		for any $n \in \mathbb{N}$.
		Using the same arguments as above we have
		\begin{equation}\label{4.17}
			\begin{aligned}
				\sup\limits_{\tau \in \mathbb{R}} \mathbb{E}|g^{t_n}(\tau)- \tilde{g}(\tau)|^2  \leq 4 C_g^2
			\end{aligned} \notag
		\end{equation}
		for any $n \in \mathbb{N}$. Then
		\begin{align*}
			\sup\limits_{|t| < T} \mathbb{E}|\xi_n|^2 &\leq 4N^2 \left[ \frac{2}{{\alpha}^2} {\left(1 - e^{-{\alpha}(T + \tilde{t})}\right)}^2 \left( L^2 \mathbb{E} \big| \xi_n \big|^2 + \sup\limits_{|t| \leq \tilde{t}} \mathbb{E} F^2_{n,\tau} \right)  
			+  \frac{4}{{\alpha}^2} {e^{-2{\alpha}(\tilde{t} - T)}} (C_f + LR)^2 \right.
			\\ & \quad + \left. C(\alpha,H) \left( \sup\limits_{|t| \leq \tilde{t}} \mathbb{E} G^2_{n,\tau} + e^{-2{\alpha}(t  + \tilde{t})} C_g^2 \right) \right].
		\end{align*}
		By organizing we obtain
		\begin{equation}\label{36}
			\begin{aligned}
				&\quad \left( 1 - \frac{8N^2 L^2}{{\alpha}^2} {(1 - e^{-{\alpha}(T + \tilde{t})})}^2 \right)\sup\limits_{|t| < T} \mathbb{E} |\xi_n|^2
				\\ &\leq  \frac{8N^2}{{\alpha}^2} \sup\limits_{|t| < \tilde{t}} \mathbb{E} F^2_{n,\tau}
				+ \frac{16N^2}{{\alpha}^2} {e^{-2{\alpha}(\tilde{t} - T)}} (C_f + LR)^2
				+ 4N^2 C(\alpha,H) \left( \sup\limits_{|t| \leq \tilde{t}} \mathbb{E} G^2_{n,\tau} + e^{-2{\alpha}(t  + \tilde{t})} C_g^2\right).
			\end{aligned}
		\end{equation}
		Under the assumption $L < \frac{{\alpha}}{2\sqrt{2}N}$, we have
		\begin{displaymath}
			\begin{aligned}
				\left( 1 - \frac{8N^2 L^2}{{\alpha}^2} {(1 - e^{-{\alpha}(T + \tilde{t})})}^2 \right) > 0.
			\end{aligned}
		\end{displaymath}
		Let now $\left\{ \tilde{t}_n \right\}$ be a sequence of positive numbers such that $\tilde{t}_n \to \infty$ as $n \to \infty$. By inequality $\eqref{36}$ we have
		\begin{equation}\label{4.18}
			\begin{aligned}
				&\quad \left( 1 - \frac{8N^2 L^2}{{\alpha}^2} {(1 - e^{-{\alpha}(T + \tilde{t}_n)})}^2 \right)\sup\limits_{|t| < T} \mathbb{E} |\xi_n|^2
				\\ &\leq  \frac{8N^2}{{\alpha}^2}  \sup\limits_{|t| < \tilde{t}} \mathbb{E} F^2_{n,\tau}
				+ \frac{16N^2}{{\alpha}^2} {e^{-2{\alpha}(\tilde{t}_n - T)}} (C_f + LR)^2
				+ C(\alpha,H,N) \left( \sup\limits_{|t| \leq \tilde{t}_n} \mathbb{E} G^2_{n,\tau} + e^{-2{\alpha}(t  + \tilde{t}_n)} C_g^2\right).
			\end{aligned}
		\end{equation}
		Note that the families  $ \left\{ F_{n,\tau}^2 : n \in \mathbb{N}, \tau \in \mathbb{R} \right\} $ and $ \left\{ G_{n,\tau}^2 :n \in \mathbb{N}, \tau \in \mathbb{R} \right\}$ are uniformly integrable, and the constant $ C(\alpha,H,N) $ is independent of $ n $. Together with $\eqref{4.12}$, $\eqref{4.13}$ and Remark $\ref{remark 2.2}$ , this implies that by taking the limit in $\eqref{4.18}$ as $n \to \infty$, we obtain
		\begin{equation}
			\begin{aligned}
				\lim\limits_{n \to \infty} \max\limits_{|t| < T} \mathbb{E}|\xi_n|^2 = 0
			\end{aligned}
		\end{equation}
		for any $T > 0$. That is, $x_n(t) \to \tilde{x}(t)$ in distribution uniformly in $t \in [-T, T]$ for any $T > 0$ as $n \to \infty$. In addition, we can show that $x_n(t)$ and $x(t + t_n)$ share the same distribution, thus $\left\{ t_n \right\} \in \tilde{\mathfrak{M}}_{x}$, and $x$ is uniformly compatible in distribution. The theorem is completely proved.
	\end{proof}
	\begin{corollary}\label{corollary 2}
		Assume that the conditions of Theorem $\ref{theorem 4.5}$ hold.
		\begin{itemize}
			\item[1.] If the functions $f$ and $g$ are jointly stationary ${\rm (}$respectively, $\tau-$periodic, quasi-periodic with the spectrum of frequencies $v_1, v_2,..., v_k$, Bohr almost periodic, Bohr almost automorphic, Birkhoff recurrent, Lagrange stable, Levitan almost periodic, almost recurrent, Poisson stable${\rm )}$ in $t \in \mathbb{R}$ uniformly with respect to $x \in \mathbb{H}$ on every bounded subset, then so is the unique bounded solution $\xi$ of equation $\eqref{4.1}$ in distribution.
			\item[2.]  If $f$ and $g$ are jointly pseudo-periodic ${\rm (}$respectively, pseudo-recurrent${\rm )}$ and $f$ and $g$ are jointly Lagrange stable, in $t \in \mathbb{R}$ uniformly with respect to $x \in \mathbb{H}$ on every bounded subset, then the unique bounded solution $\xi$ of equation $\eqref{4.1}$ is pseudo-periodic ${\rm (}$respectively, pseudorecurrent${\rm )}$ in distribution.
		\end{itemize}
	\end{corollary}
	
	\section{Convergence for Poisson stable solutions}
	In this section, we consider the following stochastic semi-linear evolution equation:
	\begin{equation}\label{5.1}
		\begin{aligned}
			\mathrm{d} X(t)= \big( AX(t) + f(t,X(t)) \big) \mathrm{d}t + g(t)\mathrm{d}B^H_{Q}(t), 
		\end{aligned}
	\end{equation}
	where $A$ and $B^H_{Q}$ are the same as in equation $\eqref{1.1}$, and $f \in C(\mathbb{R} \times \mathbb{H}, L^2(\mathbb{P}, \mathbb{H})), g \in C(\mathbb{R}, L^0_Q(\mathbb{U}, \mathbb{H}))$ are $\mathcal{F}_t$-adapted. 
	\begin{definition}\label{definition 5.2}
		We say that an $\mathcal{F}_t$-adapted processes $\left\{x(t)\right\}_{t\geq s}$ is a mild solution of equation $\eqref{5.1}$ with initial value $x(s) = x_s (s \in \mathbb{R})$, if it satisfies the stochastic integral equation
		\begin{equation}
			\begin{aligned}
				x(t) = U(t - s)x_s + \int_{s}^{t} {U(t - \tau)f(\tau, x(\tau))} \,{\rm d}\tau + \int_{s}^{t} {U(t - \tau)g(\tau)} \,{\rm d}B^H_{Q}(\tau) , \quad t \geq s.
			\end{aligned} \notag
		\end{equation}
	\end{definition}
	\begin{theorem}\label{theorem 5.3}
		Consider equation $\eqref{5.1}$. Assume the following conditions hold:
		\begin{itemize}
			\item[1.] The semigroup $\left\{U(t)\right\}_{t\geq 0}$ generated by operator $A$ is exponentially stable;
			\item[2.] $f \in C(\mathbb{R} \times \mathbb{H}, L^2(\mathbb{P}, \mathbb{H}))$ is locally Lipschitz in $x \in \mathbb{H}$, and $ g \in C(\mathbb{R}, L^0_Q(\mathbb{U}, \mathbb{H}))$;
			\item[3.] There exists positive numbers $C_f, C_g, L \geq 0$ such that $|f(t, 0)| \leq C_f$, $ \Vert g(t)\Vert_{L^0_Q(\mathbb{U}, \mathbb{H})} \leq C_g$ for any $t \in \mathbb{R}$ and $Lip(f) \leq L$;
			\item[4.]  $L < \frac{{\alpha}}{\sqrt{6}N}$.
		\end{itemize}
		Then for any initial value $x_s$ with $\mathbb{E}|x_s|^2 < \infty$, we have
		\begin{equation}\label{5.2}
			\begin{aligned}
				\mathbb{E}|x(t; s, x_s)|^2 &\leq \frac{6N^2C_f^2 + 3N^2 C_g^2 C_H \alpha^{2-2H} {(2H - 1)^{2H-1}}}{{\alpha}^2 - 6N^2 L^2} 
				\\ & \quad + e^{-({\alpha} -  \frac{6N^2 L^2}{{\alpha}})(t - s)} \left(3N^2  \mathbb{E}|x_s|^2- \frac{6N^2C_f^2 + 3N^2 C_g^2 C_H \alpha^{2-2H} {(2H - 1)^{2H-1}}}{{\alpha}^2 - 6N^2 L^2}  \right),
			\end{aligned}
		\end{equation}
		and
		\begin{equation}\label{5.3}
			\begin{aligned}
				\limsup\limits_{t \to \infty}\mathbb{E}|x(t; s, x_s)|^2 \leq \frac{6N^2C_f^2 + 3N^2 C_g^2 C_H \alpha^{2-2H} {(2H - 1)^{2H-1}}}{{\alpha}^2 - 6N^2 L^2}  
			\end{aligned}
		\end{equation}
		for any $t \geq s$, and uniformly with respect to $x_s$ on every bounded subset of $L^2(\mathbb{P}, \mathbb{H})$, where $x(t; s, x_s)$ denotes the solution of equation $\eqref{5.1}$ passing through $x_s$ at the initial moment $s$.
	\end{theorem}
	\begin{proof}
		By the Cauchy-Schwarz inequality and lemma $\ref{lemma 2.4}$ we have
		\begin{align*}
			\mathbb{E}|x(t; s, x_s)|^2 &= \mathbb{E}\bigg|U(t - s)x_s + \int_{s}^{t} {U(t - \tau)f(\tau, x(\tau; s, x_s))} \,{\rm d}\tau 
			+ \int_{s}^{t} {U(t - \tau)g(\tau)} \,{\rm d}B^H_{Q}(\tau)\bigg|^2
			\\ &\leq 3N^2 \left(  e^{-2{\alpha}(t - s)}|x_s|^2 
			+  \int_{s}^{t} {e^{-{\alpha}(t - \tau)} }\,{\rm d}\tau \int_{s}^{t} {e^{-{\alpha}(t - \tau)} \mathbb{E}|f(\tau, x(\tau; s, x_s))}|^2 \,{\rm d}\tau \right.
			\\ &\quad + \left. C_H \sum_{n = 1}^{\infty} \bigg( \int_{-\infty}^{t} \big| {e^{-{\alpha}(t - \tau)}g(\tau) Q^{\frac{1}{2}}e_n \big|^{\frac{1}{H}} } \,{\rm d}\tau  \bigg)^{2H} \right)
			\\ &\leq 3N^2 \left(  e^{-2{\alpha}(t - s)}\mathbb{E}|x_s|^2 + \frac{1}{{\alpha}} \int_{s}^{t} {e^{-{\alpha}(t - \tau)} \mathbb{E}|f(\tau, x(\tau; s, x_s))}|^2 \,{\rm d}\tau \right.
			\\ &\quad + \left. C_H \left( \int_{-\infty}^{t} {\big| {e^{-\frac{{\alpha}}{2}(t - \tau)}} \big|^{\frac{2}{2H-1}} } \,{\rm d}\tau  \right)^{2H - 1} 
			\int_{-\infty}^{t}  { e^{-{\alpha}(t - \tau)}\mathbb{E} \Vert g(\tau)\Vert^2_{L^0_Q(\mathbb{U}, \mathbb{H})}} \,{\rm d}\tau  \right)
			\\ &\leq 3N^2 \left(  e^{-2{\alpha}(t - s)}\mathbb{E}|x_s|^2 + \frac{1}{{\alpha}} \int_{s}^{t} {e^{-{\alpha}(t - \tau)} \mathbb{E}|f(\tau, x(\tau; s, x_s))}|^2 \,{\rm d}\tau \right.
			\\ &\quad + \left. C_H \left(\frac{2H - 1}{\alpha} \right)^{2H - 1} \int_{-\infty}^{t}  { e^{-{\alpha}(t - \tau)}\mathbb{E} \Vert g(\tau)\Vert^2_{L^0_Q(\mathbb{U}, \mathbb{H})}} \,{\rm d}\tau  \right)
			\\ &\leq 3N^2 e^{-{\alpha}t} \left(  e^{{\alpha}s}\mathbb{E}|x_s|^2 + \left(\frac{2 C_f^2}{{\alpha}} + C_g^2 C_H \left(\frac{2H - 1}{\alpha} \right)^{2H - 1} \right) \int_{s}^{t} {e^{{\alpha} \tau} } \,{\rm d}\tau \right. 
			\\ &\quad + \left. \frac{2 L^2}{{\alpha}} \int_{s}^{t} {e^{{\alpha} \tau}  \mathbb{E}|x(\tau; s, x_s)|^2} \,{\rm d}\tau \right).
		\end{align*}
		Thus,
		\begin{equation}\label{5.4}
			\begin{aligned}
				\mathbb{E}|x(t; s, x_s)|^2 &\leq 3N^2 e^{-{\alpha}t} \left(  e^{{\alpha}s}\mathbb{E}|x_s|^2 + \left(\frac{2 C_f^2}{{\alpha}} + C_g^2 C_H \left(\frac{2H - 1}{\alpha} \right)^{2H - 1} \right) (e^{{\alpha} t} -e^{{\alpha} s})\right.
				\\ &\quad + \left. \frac{2 L^2}{{\alpha}} \int_{s}^{t} {e^{{\alpha} \tau}\mathbb{E}|x(\tau; s, x_s)|^2} \,{\rm d}\tau \right).
			\end{aligned}
		\end{equation}
		Denote
		\begin{equation}
			\begin{aligned}
				u(t) := e^{{\alpha}t} \mathbb{E}|x(t; s, x_s)|^2, \quad  t \geq s.
			\end{aligned} \notag
		\end{equation}
		Substituting into $\eqref{5.4}$, we have
		\begin{equation}\label{5.5}
			\begin{aligned}
				u(t) &\leq 3N^2   e^{{\alpha}s}\mathbb{E}|x_s|^2 + 3N^2 \left(\frac{2 C_f^2}{{\alpha}} + C_g^2 C_H \left(\frac{2H - 1}{\alpha} \right)^{2H - 1} \right) (e^{{\alpha} t} -e^{{\alpha} s})
				\\ & \quad + \frac{6N^2 L^2}{{\alpha}} \int_{s}^{t} {e^{{\alpha} \tau}\mathbb{E}|x(\tau; s, x_s)|^2} \,{\rm d}\tau
				\\ & := 3N^2   e^{{\alpha}s}\mathbb{E}|x_s|^2 + C_1 (e^{{\alpha} t} -e^{{\alpha} s}) 
				+ C_2 \int_{s}^{t} {u(\tau)} \,{\rm d}\tau,
			\end{aligned}
		\end{equation}
		where $C_1 := C_1(\alpha,H,N) = 3N^2 \left(\frac{2 C_f^2}{{\alpha}} + C_g^2 C_H \left(\frac{2H - 1}{\alpha} \right)^{2H - 1} \right)$ and $C_2 := C_2(\alpha,N) = \frac{6N^2 L^2}{{\alpha}}$ are positive constants. To estimate $u(t)$ we consider the equation
		\begin{equation}
			\begin{aligned}
				\hat{u}(t) &= 3N^2   e^{{\alpha}s}\mathbb{E}|x_s|^2 + C_1 (e^{{\alpha} t} -e^{{\alpha} s}) 
				+ C_2 \int_{s}^{t} {\hat{u}(\tau)} \,{\rm d}\tau,
			\end{aligned}\notag
		\end{equation}
		that is, $\hat{u}(t)$ satisfies the equation
		\begin{equation}
			\begin{aligned}
				\hat{u}^{\prime}(t) &=  C_2 \hat{u}(t) + C_1 e^{{\alpha} t}
			\end{aligned}\notag
		\end{equation}
		with initial condition $\hat{u}(s) = 3N^2   e^{{\alpha}s}\mathbb{E}|x_s|^2$. So,
		\begin{equation}
			\begin{aligned}
				\hat{u}(t) &= e^{ C_2 (t - s)} \left(  \int_{s}^{t} { C_1 e^{{\alpha}\tau -  C_2 (\tau - s)}} \,{\rm d}\tau + \hat{u}(s) \right)
				\\&= \frac{C_1}{{\alpha} - C_2} e^{{\alpha}t} + e^{{\alpha}s + C_2(t - s)} \left(3N^2  \mathbb{E}|x_s|^2- \frac{C_1}{{\alpha} - C_2} \right).
			\end{aligned}\notag
		\end{equation}
		According to comparison principle, we have $u(t) \leq \hat{u}(t)$ for any $t \geq s$, that is, for $t \geq s$
		\begin{equation}
			\begin{aligned}
				{u}(t) \leq \frac{C_1}{{\alpha} - C_2} e^{{\alpha}t} + e^{{\alpha}s + C_2(t - s)} \left(3N^2  \mathbb{E}|x_s|^2- \frac{C_1}{{\alpha} - C_2} \right).
			\end{aligned}\notag
		\end{equation}
		Therefore,
		\begin{equation}
			\begin{aligned}
				\mathbb{E}|x(t; s, x_s)|^2 \leq \frac{C_1}{{\alpha} - C_2} + e^{-({\alpha} -  C_2)(t - s)} \left(3N^2  \mathbb{E}|x_s|^2- \frac{C_1}{{\alpha} - C_2} \right).
			\end{aligned}\notag
		\end{equation}
		Under the assumptions of condition $(iv)$, we obtain $ {\alpha} -  C_2 > 0 $. Therefore, we have
		\begin{equation}
			\begin{aligned}
				\limsup\limits_{t \to \infty}\mathbb{E}|x(t; s, x_s)|^2 \leq \frac{C_1}{{\alpha} - C_2}
			\end{aligned}\notag
		\end{equation}
		uniformly with respect to $x_s$ on every bounded subset of $L^2(\mathbb{P}, \mathbb{H})$. The theorem is completely proved.
	\end{proof}
	\begin{theorem}\label{theorem 5.4}
		Consider equation $\eqref{5.1}$. Assume the following conditions hold:
		\begin{itemize}
			\item[1.] The semigroup $\left\{U(t)\right\}_{t\geq 0}$ generated by operator $A$ is exponentially stable;
			\item[2.] $f \in C(\mathbb{R} \times \mathbb{H}, L^2(\mathbb{P}, \mathbb{H}))$ is globally Lipschitz in $x \in \mathbb{H}$ and $Lip(f) \leq L$;
			\item[3.] There exists a positive constant $C_1$ such that $|f(t, 0)| \leq \sqrt{C_1}$ and $ \Vert g(t)\Vert_{L^0_Q(\mathbb{U}, \mathbb{H})} \leq \sqrt{C_1}$ for all $t \in \mathbb{R}$;
			\item[4.]  $L < \frac{{\alpha}}{2\sqrt{2}N} $.
		\end{itemize}
		Then we can obtain the following statements:
		\begin{itemize}
			\item[1.] For any $t \geq s$ and $x_1, x_2 \in L^2(\mathbb{P}, \mathbb{H})$,
			\begin{equation}\label{5.6}
				\begin{aligned}
					\mathbb{E}|x(t; s, x_1) - x(t; s, x_2)|^2 \leq 2N^2  \mathbb{E}|x_1 - x_2|^2 e^{-\left( {\alpha} - \frac{2N^2 L^2 }{{\alpha}} \right) (t - s)};
				\end{aligned}
			\end{equation}
			\item[2.] Equation $\eqref{5.1}$ has a unique solution $\varphi$ in $C_b(\mathbb{R}, L^2(\mathbb{P}, \mathbb{H}))$ which is globally asymptotically stable and
			\begin{equation}\label{5.7}
				\begin{aligned}
					\mathbb{E}|x(t; s, x_s) - \varphi(t)|^2 \leq 2N^2   \mathbb{E}|x_s - \varphi(s)|^2 e^{-\left( {\alpha} - \frac{2N^2 L^2 }{{\alpha}} \right) (t - s)}
				\end{aligned}
			\end{equation}
			for any $t \geq s$ and $x_s \in L^2(\mathbb{P}, \mathbb{H})$.
		\end{itemize}
	\end{theorem}
	\begin{proof}
		1. From Definition $\ref{definition 5.2}$, we have
		\begin{equation}
			\begin{aligned}
				x(t; s,x_i) &= U(t - s)x_i + \int_{s}^{t} {U(t - \tau)f(\tau, x(\tau; s,x_i))} \,{\rm d}\tau 
				+ \int_{s}^{t} {U(t - \tau)g(\tau)} \,{\rm d}B^H_{Q}(\tau)
			\end{aligned} \notag
		\end{equation}
		for $i = 1, 2$. By defining $y(t) := x(t; s, x_1) - x(t; s, x_2)$ for any $t \geq s$, we have
		\begin{equation}
			\begin{aligned}
				y(t) &= U(t - s)(x_1 - x_2) + \int_{s}^{t} {U(t - \tau)[f(\tau, x(\tau; s,x_1)) - f(\tau, x(\tau; s,x_2))]} \,{\rm d}\tau.
			\end{aligned} \notag
		\end{equation}
		Consequently,
		\begin{align*}
			\mathbb{E}|y(t)|^2  &\leq 2\left(  \mathbb{E}|U(t - s)(x_1 - x_2)|^2  
			+ \mathbb{E}\bigg| \int_{s}^{t} {U(t - \tau)[f(\tau, x(\tau; s,x_1)) - f(\tau, x(\tau; s,x_2))]} \,{\rm d}\tau \bigg|^2 \right)
			\\ &\leq 2N^2 \left(  \mathbb{E}| e^{-{\alpha}(t - s)}(x_1 - x_2)|^2 
			+ L^2 \mathbb{E}\bigg| \int_{s}^{t} {e^{-{\alpha}(t - \tau)} \big(x(\tau; s,x_1) - x(\tau; s,x_2)\big)} \,{\rm d}\tau \bigg|^2 \right)
			\\ &\leq 2N^2   e^{-2{\alpha}(t - s)}\mathbb{E}|x_1 - x_2|^2 
			+  2N^2 L^2 \int_{s}^{t} {e^{-{\alpha}(t - \tau)} }\,{\rm d}\tau \int_{s}^{t} {e^{-{\alpha}(t - \tau)} \mathbb{E}|\left(x(\tau; s,x_1) - x(\tau; s,x_2)\right)|^2} \,{\rm d}\tau
			\\ &\leq 2N^2   e^{-2{\alpha}(t - s)}\mathbb{E}|x_1 - x_2|^2 
			+ \frac{2N^2 L^2 }{{\alpha}} \int_{s}^{t} {e^{-{\alpha}(t - \tau)} \mathbb{E}|\left(x(\tau; s,x_1) - x(\tau; s,x_2)\right)|^2} \,{\rm d}\tau.
		\end{align*} 
		Thus,
		\begin{equation}\label{5.8}
			\begin{aligned}
				\mathbb{E}|y(t)|^2  \leq 2N^2   e^{-2{\alpha}(t - s)}\mathbb{E}|x_1 - x_2|^2  + \frac{2N^2 L^2 }{{\alpha}} \int_{s}^{t} {e^{-{\alpha}(t - \tau)} \mathbb{E}|y(t)|^2} \,{\rm d}\tau.
			\end{aligned}
		\end{equation}
		Set $z(t) := e^{{\alpha}t} \mathbb{E}|y(t)|^2$ for $t \geq s$, then from $\eqref{5.8}$ we have
		\begin{equation}
			\begin{aligned}
				z(t)  \leq 2N^2   e^{{\alpha} s}\mathbb{E}|x_1 - x_2|^2  +  \frac{2N^2 L^2 }{{\alpha}} \int_{s}^{t} {z(\tau)} \,{\rm d}\tau.
			\end{aligned}\notag
		\end{equation}
		By Gronwall-Bellman inequality,
		\begin{equation}
			\begin{aligned}
				z(t)  \leq 2N^2   e^{{\alpha} s}\mathbb{E}|x_1 - x_2|^2 e^{\frac{2N^2 L^2 }{{\alpha}} (t-s)}, \quad t \geq s,
			\end{aligned}\notag
		\end{equation}
		and consequently, by the definition of $z(t)$, we have
		\begin{displaymath}
			\begin{aligned}
				\mathbb{E}|x(t; s, x_1) - x(t; s, x_2)|^2 \leq 2N^2  \mathbb{E}|x_1 - x_2|^2 e^{-\left( {\alpha} - \frac{2N^2 L^2 }{{\alpha}} \right) (t - s)}, \quad t \geq s.
			\end{aligned}\notag
		\end{displaymath}
		
		2. According to Theorem $\ref{theorem 4.5}$, equation $\eqref{5.1}$ has a unique bounded solution $\phi \in C_b(\mathbb{R}, L^2(\mathbb{P}, \mathbb{H}))$ under the condition $L < \frac{{\alpha}}{2\sqrt{2}N} $, which is valid under the current condition $4$.
		
		Now let $\phi(t) = x(t; s, \phi(s))$ for any $t \geq s$. By applying inequality $\eqref{5.6}$ and its proof, we can conclude that Equation $\eqref{5.1}$ has a unique solution $\varphi \in C_b(\mathbb{R}, L^2(\mathbb{P}, \mathbb{H}))$, which is globally asymptotically stable. Moreover, for any $ t \geq s $ and $ x_s \in L^2(\mathbb{P}, \mathbb{H}) $, the following holds:
		\begin{displaymath}
			\mathbb{E}|x(t; s, x_s) - \varphi(t)|^2 \leq 2N^2 \mathbb{E}|x_s - \varphi(s)|^2 e^{-\left( \alpha - \frac{2N^2 L^2}{\alpha} \right) (t - s)}.
		\end{displaymath}
	\end{proof}
	
	\section{Applications}
	In this section, we validate our research findings through two concrete examples: one periodic and one quasi-periodic case. Specifically, we investigate the existence and asymptotic stability of Poisson-stable solutions for a stochastic semilinear evolution equation. By employing computational simulations, we generated visual representations to demonstrate the accuracy and reliability of our results. These simulations not only reinforce our theoretical analysis but also provide a clear, tangible illustration of the dynamic behavior of the solutions, further confirming the existence of stable solutions under both periodic and quasi-periodic conditions.
	\begin{example}
		Consider the following stochastic ordinary differential equation on $\mathbb{R}$ with Dirichlet boundary condition:
		\begin{equation}\label{49}
			\mathrm{d} u(t) = \big( -3 u(t) + \cos(t) \big) \mathrm{d}t + \sin(t) \mathrm{d}\beta^H(t),
		\end{equation}
		where $u(-100pi) = 0$ and $\beta^H(t)$ is a one-dimensional fractional Brownian motion with Hurst index $H = 0.8$.
		
		This can be rewritten as an stochastic evolution equation:
		\begin{displaymath}
			\mathrm{d} X(t) = \big( A X(t) + f(t) \big) \mathrm{d}t + g(t) \mathrm{d}B^H_Q(t),
		\end{displaymath}
		where $A = -3$, $f(t) = \cos(t)$, and $g(t) = \sin(t)$. The operator $A$ generates a $C_0$-semigroup $\{U(t)\}_{t \geq 0}$ with $\|U(t)\| \leq e^{-3t}$, implying $N = 1$ and $\alpha = 3$. Additionally, since $|f(t)| \leq 1$ and $|g(t)| \leq 1$, we have $C_f = C_g = 1$.
		
		All conditions for applying Theorems $\ref{theorem 4.3}$, $\ref{theorem 4.5}$, $\ref{theorem 5.3}$, and $\ref{theorem 5.4}$ are satisfied. Theorem $\ref{theorem 4.3}$ guarantees a unique bounded solution in $C_b(\mathbb{R})$. Moreover, as $ f(t) $ and $ g(t) $ are periodic with a period of $ 2\pi $, Theorem $\ref{theorem 4.5}$ and Corollary 2 ensure the solution $u(t)$ is periodic in distribution with a period of $ 2\pi $. Finally, Theorem $\ref{theorem 5.4}$ confirms the global asymptotic stability of $u(t)$ in the square-mean sense.
		
		To further validate our theoretical findings, we performed numerical simulations of the solution to the equation and provided graphical evidence through data visualization techniques. Specifically, we solved the equation 1,0000 times and recorded the solution values at time points $t = 0.8 \pi$, $t = 2.8 \pi$, $t = 4.8 \pi$ and $t = 6.8 \pi$. Using these values, we constructed the distribution of the solution at each of these time points, as illustrated in Figure \ref{F.1}. The resulting distributions confirmed our analytical conclusion that the solution is periodic with a period of $2\pi$ in distribution.
		\begin{figure}[htbp]
			\centering
			\includegraphics[width=15cm]{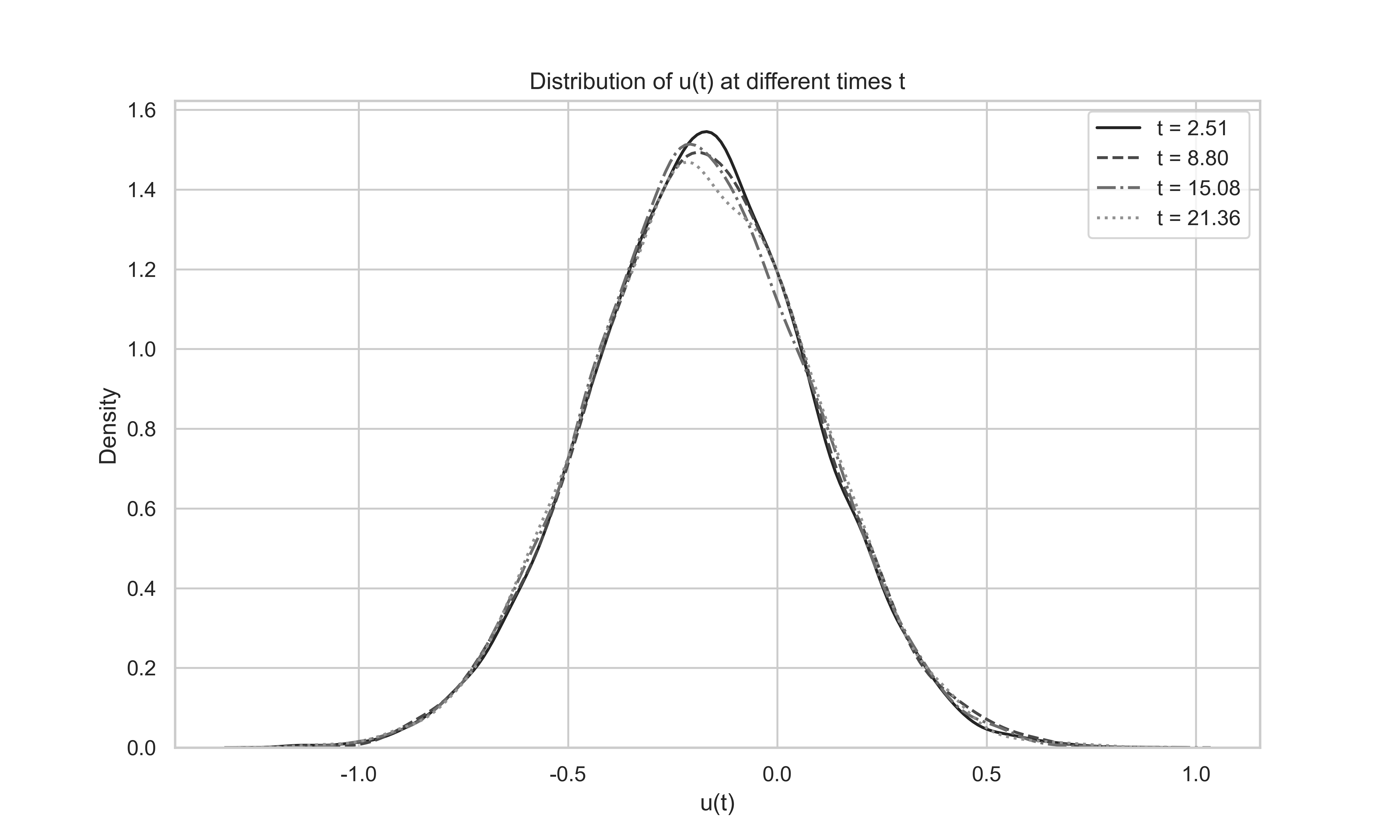}
			\centering
			\caption{The distribution of the solution to equation $\eqref{49}$ at $t = 0.8 \pi$, $t = 2.8 \pi$, $t = 4.8 \pi$ and $t = 6.8 \pi$. }
			\label{F.1}
		\end{figure} 
	\end{example}
	\begin{example}
		Let $\mathbb{H} := L^2(0, 1)$ and the norm on $\mathbb{H}$ by $\Vert \cdot \Vert$. Consider the following stochastic heat equation on the interval $[0,1]$ with Dirichlet boundary
		condition:
		\begin{equation}\label{50}
			\begin{aligned}
				\mathrm{d} u(t,x)= \big( \Delta u + \frac{(\sin(t) + \cos(\sqrt{3}t))\sin(u)}{3} \big) \mathrm{d}t + \frac{\cos(t) + \sin(\sqrt{2}t)}{2} \mathrm{d}\beta^H(t), 
			\end{aligned}
		\end{equation}
		\begin{align*}
			u(t,0) = 0.1 \sin(t),\quad u(t,1) = 0.1 \cos(t), \quad t \geq 0,
		\end{align*}
	where $\Delta :H^2(0,1) \cap H^1_0(0,1) \to L^2(0,1)$ is a Laplace operator which generates a $C_0$-semigroup $\left\{ U(t) \right\}_{t \geq 0}$ acting on ${\mathbb{H}}$.  $\beta^H(t)$ is a one-dimensional fBm. We can write this stochastic heat equation as an evolution equation:
	\begin{displaymath}
			\mathrm{d} X(t)= \big( AX(t) + f(t,X(t)) \big) \mathrm{d}t + g(t)\mathrm{d}B^H_{Q}(t), 
	\end{displaymath}
	where
	\begin{displaymath}
		\quad A := \Delta, \quad f(t,X(t)) := \frac{(\sin(t) + \cos(\sqrt{3}t))\sin(X)}{3},  \quad g(t) := \frac{\cos(t) + \sin(\sqrt{2}t)}{2}.
	\end{displaymath}
 	Note that the operator $A$ generates a $C_0$-semigroup $\{ U(t) \}_{t \geq 0}$ satisfying $\|U(t)\| \leq e^{-\pi^2 t}$ for all $t \geq 0$, i.e., $N = 1$ and $\alpha = \pi^2$. Next, we proceed to validate the conditions for $f(t, X)$ and $g(t)$. Specifically, note that $|f(t, 0)| \leq 1$ and $\text{Lip}(f) \leq \frac{2}{3}$, i.e., $C_f = C_g = 1$ and $L = \frac{2}{3}$. It is straightforward to verify that conditions $(C1)$–$(C3)$ are satisfied. Moreover, we can confirm that $L < \frac{\alpha}{2\sqrt{2}N}$. Therefore, all the conditions necessary for the applicability of Theorems \ref{theorem 4.3}, \ref{theorem 4.5}, \ref{theorem 5.3}, and \ref{theorem 5.4} are fulfilled.
 	
 	Furthermore, it is observed that $f(t,X(t)), g(t)$ exhibit quasi-periodicity in $t$ uniformly with respect to $X \in \mathbb{H}$. According to Theorem $\ref{theorem 4.3}$, equation $\eqref{50}$ admits a unique bounded solution in $C_b(\mathbb{R}, L^2(\mathbb{P}, \mathbb{H}))$. Moreover, by Theorem $\ref{theorem 4.5}$ and corollary $\ref*{corollary 2}$ the solution $u(t,x)$ is uniformly compatible in distribution, i.e. the solution $u(t,x)$ is quasi-periodic in $t$ uniformly with respect to $x \in \mathbb{H}$. Additionally, Theorem $\ref{theorem 5.4}$ confirms the global asymptotic stability of the solution $u(t,x)$ in the square-mean sense.
 
	To further corroborate our theoretical findings, we conduct numerical simulations of the solution to the equation and utilize data visualization techniques to furnish graphical evidence, as delineated in Fig. $\ref{F.2}$, thereby affirming the precision of our analytical results.
	\begin{figure}[htbp]
		\centering
		\includegraphics[trim={0.6cm 1.1cm 1.8cm 1.3cm}, clip, width=16cm]{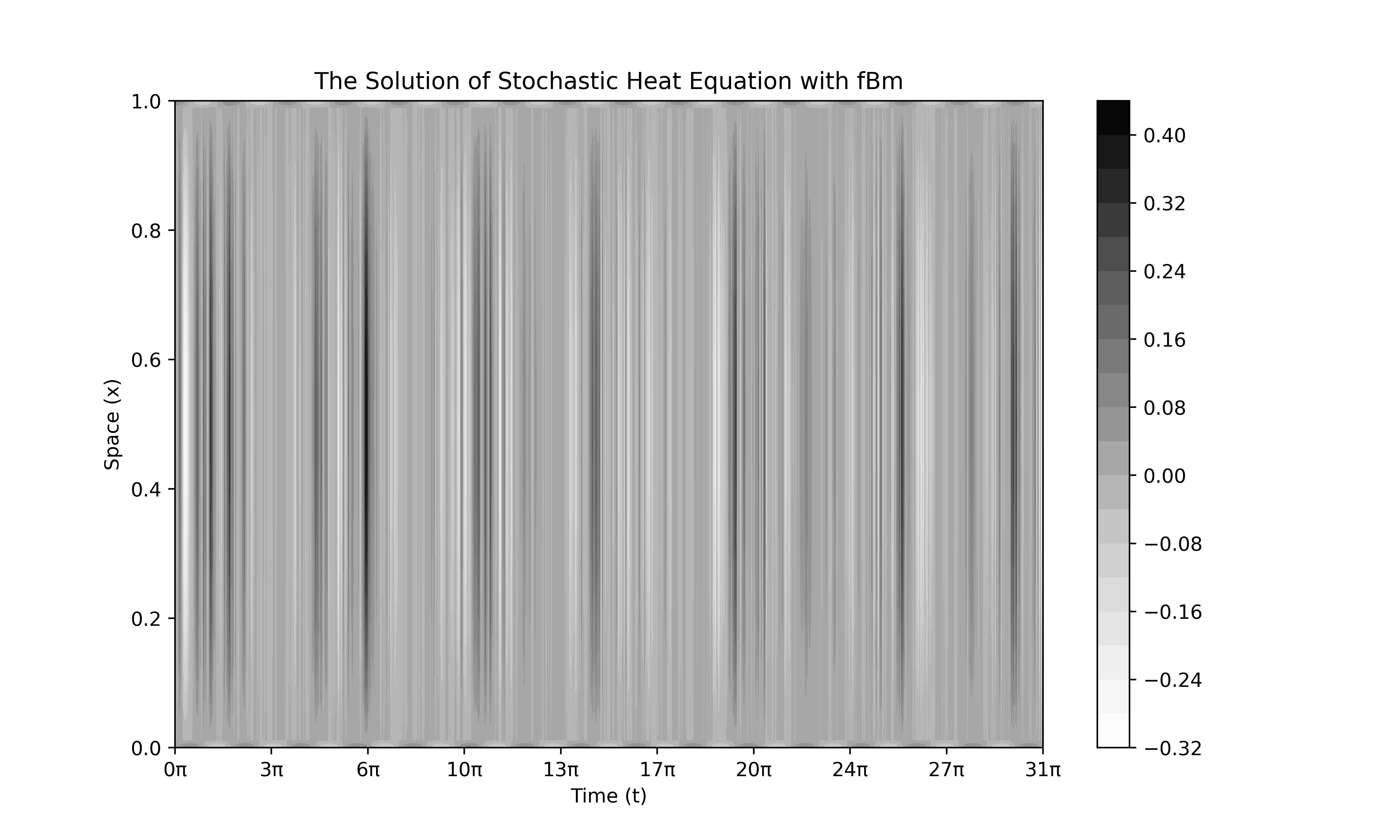}
		\centering
		\caption{The graph of the solution to equation $\eqref{50}$ on $t \in [0,31 \pi], x \in [0,1]$.}
		\label{F.2}
	\end{figure} 
	\end{example}

	
	
	\bibliographystyle{plain}
	\bibliography{Ref}
	
\end{document}